\documentclass[12pt,reqno]{amsart}
\usepackage[top=1in, bottom=1in, left=1in, right=1in]{geometry}
\usepackage{times, amsthm, amssymb, amsmath, amsfonts, amsthm, bm, graphicx, mathrsfs}
\usepackage[all]{xy}
\usepackage[usenames,dvipsnames]{color}
\usepackage[pagebackref=true,pdftex]{hyperref}

\newcommand{\excise}[1]{}

\newtheorem{theorem}{Theorem}[section]
\newtheorem{lemma}[theorem]{Lemma}

\newtheorem{cor}[theorem]{Corollary}
\newtheorem{prop}[theorem]{Proposition}
\newtheorem{conjecture}[theorem]{Conjecture}

\theoremstyle{definition}
\newtheorem{example}[theorem]{Example}
\newtheorem{remark}[theorem]{Remark}

\newtheorem{notation}[theorem]{Notation}

        {\begin{list}
                {\noindent\makebox[0mm][r]{\arabic{enumi}.}}
                {\leftmargin=5.5ex \usecounter{enumi}}
        }
        {\end{list}}

        {\begin{list}
                {\noindent\makebox[0mm][r]{(\roman{enumi})}}
                {\leftmargin=5.5ex \usecounter{enumi}}
        }
        {\end{list}}

\DeclareMathAlphabet{\mathpzc}{OT1}{pzc}{m}{it}
\usepackage{calligra}
\DeclareMathAlphabet{\mathcalligra}{T1}{calligra}{m}{n}

\def\<{\langle}
\def\>{\rangle}
\def\0{\mathbf{0}}
\def\AA{{\mathbb A}}
\def\CC{{\mathbb C}}
\def\cC{{\mathcal C}}
\def\DD{{\mathcal D}}
\def\EE{{\mathbb E}}
\def\cE{{\mathcal E}}

\def\HH{{\mathcal H}}
\def\ds{{\operatorname{d}_s}}

\def\KK{{\mathcal K}}

\def\MM{{\mathcal M}}
\def\NN{{\mathbb N}}
\def\cO{{\mathcal O}}
\def\PP{{\mathbb P}}
\def\QQ{{\mathbb Q}}

\def\RR{{\mathbb R}}
\def\TT{{\mathbb T}}

\def\ZZ{{\mathbb Z}}

\def\bbE{{\mathbb E}}

\def\cJ{{\mathcal J}}

\def\fS{{\mathfrak S}}
\renewcommand\SS{{\mathbb S}}
\renewcommand\Re{{\mathrm {Re} }}

\def\del{\partial}
\def\dx#1{{\partial_{x_{#1}}}}

\def\boldone{\boldsymbol{1}}
\def\boldzero{\boldsymbol{0}}
\def\ini{{\mbox{in}}}

\def\qdeg{{\rm qdeg}}

\def\gr{{\rm gr}}
\def\Hom{\operatorname{Hom}}
\def\Ext{\operatorname{Ext}}
\def\vol{{\rm \operatorname{vol}}}
\def\conv{{\rm \operatorname{conv}}}
\def\rank{{\operatorname{rank}}}
\def\Var{{\rm Var}}
\def\Spec{{\rm Spec}}
\def\codim{{\rm codim}}
\def\ns{\operatorname{nsupp}}
\def\charVar{{\operatorname{Char}}}

\def\charC{{\operatorname{CC}}}
\def\Var{{\rm Var}}
\def\Irr{{\operatorname{Irr}}}
\newcommand\HA{{H_A(\beta)}}

\def\minus{\smallsetminus}
\def\nothing{\varnothing}

\newcommand*{\defeq}{\mathrel{\vcenter{\baselineskip0.5ex \lineskiplimit0pt
                     \hbox{\scriptsize.}\hbox{\scriptsize.}}}%
                     =}

\def\ol#1{{\overline {#1}}}
\def\wt#1{{\widetilde {#1}}}
\def\dlim{{\lim\limits_{\raisebox{.2ex}{$\scriptstyle\longrightarrow$}}}{_{\,}}}

\numberwithin{equation}{section}
\newcommand{\christine}[1]{{\color{blue} \sf $\clubsuit$ C: #1}}
\newcommand{\maria}[1]{{\color{ForestGreen} \sf $\clubsuit$ M: #1}}

\begin{document}

\mbox{}
\title[Characteristic cycles and Gevrey series solutions of $A$-hypergeometric systems]{Characteristic cycles and Gevrey series solutions\\ of $A$-hypergeometric systems}

\author{Christine Berkesch}
\address{School of Mathematics \\
University of Minnesota.}
\email{cberkesc@umn.edu}

\author{Mar\'ia-Cruz Fern\'andez-Fern\'andez}
\address{Departamento de \'Algebra \\
Universidad de Sevilla.}
\email{mcferfer@algebra.us.es}

\thanks{CB was partially supported by NSF Grants DMS 1661962, DMS 1440537, OISE 0964985. MCFF was partially supported by MTM2016-75024-P, PP2014-2397, P12-FQM-2696 and FEDER}

\subjclass[2010]{13N10, 32C38, 33C70, 14M25.}
\keywords{$A$--hypergeometric system, toric ring, $D$--module, characteristic cycle, irregularity sheaf, Gevrey series.}

\begin{abstract}
We compute the $L$-characteristic cycle of an $A$-hypergeometric system and higher Euler--Koszul homology modules of the toric ring. 
We also prove upper semicontinuity results about the multiplicities in these cycles and apply our results to analyze the behavior of  Gevrey solution spaces of the system. 
\end{abstract}
\maketitle

\mbox{}
\vspace{-15mm}
\parskip=0ex
\parindent2em
\parskip=1ex
\parindent0pt

\setcounter{section}{0}
\section*{Introduction}

Let $D$ denote the Weyl algebra on $X = \CC^n$ with coordinates $x = x_1,\dots,x_n$. Let $\del_i$ denote the variable that acts on $\CC[x]$ as $\frac{\del}{\del x_i}$ and write $\del = \del_1,\dots,\del_n$. 
A \emph{weight vector} on $D$ is $L=(L_x,L_\del)\in\QQ^n\times\QQ^n$ such that $L_x+L_\del\geq 0$.
Such a vector induces an exhaustive increasing filtration $L$ on $D$ by, for $k \in \QQ$, 
\[
L^k D \defeq \CC\cdot\{x^u\del^v \mid L\cdot(u,v) \leq k\}.
\] 
Write $L^{<k} D\defeq \bigcup_{\ell < k} L^{\ell} D$.   
For any $P$ in $L^k D\minus L^{<k} D$, set 
\[
\ini_L(P)\defeq P+L^{<k}D\in\gr^{L,k} D \defeq
L^k D / L^{<k} D \subseteq \gr^L D 
\quad\text{and}\quad \deg^L(P) \defeq k.
\]  
For a left $D$-ideal $I$ and the $D$-module $M=D/I$, set 
\[
\gr^L(I)\defeq \<\ini_L(P)\mid P\in I\>\subseteq \gr^L(D)
\quad\text{and}\quad 
\gr^L(M)=\gr^L(D)/\gr^L(I).
\] 
If $L_x+L_\del=0$, the associated graded ring $\gr^L D$ is isomorphic to $D$ and $\gr^L(I)$ can be identified with a left $D$-ideal, which is also called a \emph{Gr\"obner deformation} of $I$ in \cite{SST}. It is suggestive to call $\gr^L (M)$ the Gr\"obner deformation of $M$ with respect to $L$. On the other hand, if $L_x+L_\del>0$, the associated graded
ring $\gr^L D$ is isomorphic to the coordinate ring of $T^*X\cong
\CC^{2n}$, which is a polynomial ring in $2n$ variables. In this latter case, 
the \emph{$L$-characteristic variety} of $M$ is
\vspace{-1mm}
\begin{equation}
\label{eqn:charVar}
\charVar^L(M) \defeq \Var(\gr^L(I)) \subseteq T^*X\cong\CC^{2n}.
\end{equation}
The \emph{$L$-characteristic cycle} of $M$ is the finite formal sum 
\[
\charC^L (M) \defeq \sum_{C} \mu^{L,C}(M) \cdot C,
\] 
where $C$ runs over the irreducible components of $\charVar^L(M)$, and 
\[
\mu^{L,C}(M)\defeq \ell ( (\gr^L (M))_{P_C})
\] 
is the multiplicity of $\gr^L (M)$ along $C$, where $P_C$ is the defining ideal of $C$ in $\gr^L (D)$ and $\ell$ denotes the length of a $\gr^L (D)_{P_{C}}$-module.   

The projective weight vector $F = (\boldzero_n,\boldone_n) \defeq (0,\dots,0,1,\dots,1)\in\QQ^n\times\QQ^n$ induces the \emph{order filtration} on $D$. 
We notice that $\charVar^F (M)$ and $\charC^F (M)$ are called, respectively, the \emph{characteristic variety} and the  \emph{characteristic cycle} of $M$.
If $M$ is \emph{holonomic}, that is, the dimension of its characteristic variety is $n$, then the \emph{rank} of $M$, defined $\rank(M)\defeq\dim_{\CC(x)} \CC(x)\otimes_{\CC[x]} M$, coincides with the dimension of the space of germs of its holomorphic solutions at any nonsingular point by a result of Kashiwara (see e.g.  ~\cite[Theorem~1.4.19]{SST}). Notice that $\rank(M)=\mu^{F,C}(M)$ for $C=T^*_X X$.

One motivation for the study of $L$-characteristic cycles comes from the theory of irregularity of holonomic $D$-modules. 
For a flavor of this deep and involved theory that fits the goals of this paper, 
a projective weight vector of the form  $L=F+(s-1)V$ where $s\in \QQ$ and $V=(-w,w)$, with $w=(0,\ldots,0,1)$, induces the Kashiwara--Malgrange filtration along the coordinate hyperplane $Y=\{x_n=0\}\subseteq X= \CC^n$. In this case, the $L$-characteristic variety $\charVar^{F+(s-1)V}(M)$ is locally constant with respect to $s\in\QQ$, except for at a finite set of values called \emph{algebraic slopes} of $M$ along $Y$. This is a global version of the algebraic slopes defined and studied by Laurent \cite{Laurent}. On the other hand, the \emph{analytic slopes} of $M$ along $Y$ were defined as jumps in the \emph{Gevrey} filtration of the \emph{irregularity sheaf} of $M$ along $Y$ by Mebkhout \cite{Mebkhout}. The comparison theorem for slopes states that the algebraic and analytic slopes for $M$ along $Y$ coincide, and, even more, the Euler--Poincar\'e characteristic of the irregularity sheaf can be computed in terms of the $L$-characteristic cycles of $M$ \cite{LM}. In particular, certain multiplicities in the $L$-characteristic cyles are closely related to the dimension of the space of Gevrey solutions of $M$ along $Y$.

Another motivating idea of this article is that the $F$-characteristic cycle of a Gr\"obner deformation of a holonomic $D$-module $M$ is equal to the $L$-characteristic cycle of $M$ for an approppriate $L$ (see Lemma \ref{characteristic-cycles} for the precise statement). In particular, the holonomic rank of such a Gr\"obner deformation is the multiplicity of the component $T^*_{X}X$ in $\charC^L(M)$. 

Our main interest is $A$-hypergeometric $D$-modules, also known as GKZ-systems after their introduction and study by Gelfand, Graev, Kapranov, and Zelevinsky \cite{GGZ,GKZ,GKZ-int}. Let $A =[a_{ij}] = [a_1 \cdots a_n]\in\ZZ^{d\times n}$ be an integral matrix such that the group generated by the columns of $A$, $\ZZ A$, is equal to $\ZZ^d$, and the positive real cone $\RR_{\geq 0} A$ over the columns is pointed. 
Let 
\[
I_A \defeq \<\del^u-\del^v\mid Au=Av\>\subseteq\CC[\del]
\] 
denote the \emph{toric ideal} of $A$. For $\beta\in\CC^d$, write $E - \beta$ for the sequence of \emph{Euler operators} given by 
\[
E_i - \beta_i \defeq \sum_{j=1}^n a_{ij}x_j\del_j - \beta_i
\] 
for $i=1,\dots,d$. 
The \emph{$A$-hypergeometric system} of $A$ at $\beta\in\CC^d$ is 
\[
\HA \defeq 
D\cdot( I_A + \<E-\beta\> )
\quad\text{with associated module}\ \  
M_A(\beta) \defeq D/\HA.
\]

A weight vector $L=(L_x,L_{\del})\in\QQ^n\times\QQ^n$ as above is called \emph{projective} if $L_x+L_{\del} =
c\cdot\boldone_n \defeq c\cdot(1,\dots,1)$ for some constant $c>0$. Notice that any Euler operator $E_i$ is homogeneous with respect to such a filtration. In~\cite{slopes}, the irreducible components of $\charC^L(M_A(\beta))$ were enumerated, and when $\beta$ is generic (or not rank--jumping), $\charC^L(M_A(\beta))$ was computed. In this article, we compute $\charC^L(M_A(\beta))$ for any $\beta$, along with the characteristic cycles of higher Euler--Koszul homology modules (see Section~\ref{sec:EKhom}) of the toric ring $\CC[\del]/I_A$.  We also provide upper semicontinuity results for some of these multiplicities and apply our results to the Gevrey solution spaces of $M_A(\beta)$.

\subsection*{Outline}
In \S\S\ref{sec:EKhom}-\ref{sec:charC}, we provide background and preliminary results on Euler--Koszul homology and $L$-characteristic cycles of $A$-hypergeometric systems. 
We compute the multiplicities in the characteristic cycles of the Euler--Koszul homology of the toric ring in \S\ref{sec:mults}, with consequences in \S\ref{sec:consequences}. 
We provide upper semicontinuity results in \S\ref{sec:upperSC} and study Gevrey solutions of $\HA$ in \S\ref{sec:gevrey}. 

\subsection*{Acknowledgements}
We are grateful to Francisco Jes\'us Castro Jim\'enez, Laura Felicia Matusevich, and Uli Walther for helpful conversations related to this work. 
The second author would like to thank the School of Mathematics of the University of Minnesota for the hospitality during her visit to work on this paper with the first author. 

\section{Euler--Koszul homology}
\label{sec:EKhom}

In this section, we present background related to Euler--Koszul homology, as found in~\cite{MMW,ekdi}, with some additions needed in the sequel. 
We use the convention that $0\in\NN$. 
Recall that $a_i$ denotes the $i$th column of the matrix $A$. 
Given a subset $\tau\subseteq A$ of the column set of $A$, the semigroup generated by $\tau$, 
\[
\NN \tau \defeq \left\{ \sum_{a_i\in\tau} j_i a_i\; \Bigg|\; j_i\in\NN \text{ for all } a_i\in \tau\right\},
\]
generates the semigroup ring $S_\tau\defeq \CC[\NN \tau]$. 
With $\pi_\tau\colon \CC[\del_\tau]\defeq\CC[\del_j\,|\, a_j\in \tau]\to S_\tau$ denoting the map induced by $\tau$, 
we have the isomorphism of rings $S_\tau  \cong \CC[\del_\tau]/\ker\pi_\tau$.
When convenient, we will abuse notation and also view $\tau$ as a matrix.

A subset $G$ of the columns of the matrix $A$ is a \emph{face} of $A$, denoted $G\preceq A$, if $\RR_{\geq0}G$ is a face of the cone $\RR_{\geq0}A$ and $G = A\cap\RR G$. 
The codimension of a nonempty face $G$ is $\codim(G)\defeq d-\dim (\RR G)$, with $\codim (\nothing)=d$ by convention. Let $G^c$ denote the complement of $G$ in $A$.

Define a $\ZZ^d$-grading on $D$ via 
$\deg(x_i)\defeq -a_i$ and $\deg(\dx{i})\defeq a_i$.
A $\ZZ^d$-graded $\CC[\del]$-module $N$ is \emph{toric} if it has a filtration 
\[
0 = N^{(0)} \subseteq N^{(1)} \subseteq \cdots \subseteq N^{(\ell - 1)} \subseteq N^{(\ell)} = N
\]
such that for each $i$, $N^{(i)}/N^{(i-1)}$ is a $\ZZ^d$-graded translate of 
$S_{G_i}$ for some face $G_i \preceq A$. 
The \emph{degree set} of a finitely generated $\ZZ^d$-graded $\CC[\del]$-module $N$ is $\deg(N) \defeq \{\alpha\in\ZZ^d\mid N_\alpha\neq 0\}$. 
The \emph{quasidegree set} of $N$, denoted $\qdeg(N)$, is the Zariski closure of
$\deg(N)$ under the natural embedding $\ZZ^d\hookrightarrow\CC^d$. 
A $\ZZ^d$-graded $\CC[\del]$-module $N$ is \emph{weakly toric} if 
there is a filtered partially ordered set $(\fS,\leq)$
and a $\ZZ^d$-graded direct limit 
\[
\phi_s\colon N^{(s)} \to {\dlim}_{s\in\fS} N^{(s)} = N,
\]
where $N^{(s)}$ is a toric $\CC[\del]$-module for each $s\in\fS$. 
The \emph{quasidegrees} of $N$ are
\[
\qdeg(N)\defeq \bigcup_{s\in\fS} \qdeg(\phi_s(N^{(s)})),
\]
where each $\qdeg(\phi_s(N^{(s)}))$ is already defined since $\phi_s(N^{(s)})$ is toric for each $s$. 

Let $N$ be a weakly toric module. 
Given a homogeneous $y\in D\otimes_{\CC[\del]} N$, define an action of the Euler operators for $1\leq i\leq d$ by 
\begin{eqnarray*}
(E_i-\beta_i)\circ y=(E_i - \beta_i + \deg_i(y))y,
\end{eqnarray*}
and extend this action $\CC$-linearly to $D\otimes N$.   
With this sequence of commuting endomorphisms on $D\otimes N$, let $\KK^A_\bullet(N,\beta)$ denote the Koszul complex on the left $D$-module $D\otimes_{\CC[\del]} N$, which we call the \emph{Euler--Koszul complex} of $N$ at $\beta$. 
Its homology is denoted $\HH^A_i(N,\beta)\defeq H_i(\KK^A_\bullet(N,\beta))$ or simply $\HH_i(N,\beta)$ when $A$ is clear from the context. 
Euler--Koszul homology was first introduced in~\cite{MMW} for toric modules and extended to weakly toric modules in~\cite{ekdi}.

If $b\in\ZZ^d$, we denote by $N(b)$ a $\ZZ^d$-graded translated copy of $N$ such that $N(b)_v=N_{v-b}$ for all $v\in\ZZ^d$. Thus, $\deg (N(b)) =b+\deg (N)$. For example, if $N=S_A=\CC[\NN A]$ then $N(b)=\CC[\NN A]t^b$. 
Euler--Koszul homology is compatible with these graded shifts. Namely, we have
\begin{align}
\label{eqn:shiftHH}
\HH_q(N(b),\beta)\cong \HH_q(N,\beta-b)(b).
\end{align}

\begin{theorem}{\cite[Theorem~5.4]{ekdi}}
\label{thm:HHvanishing}
For a weakly toric module $N$, the following are equivalent:\\
\vspace{-7mm}
\begin{enumerate}
\item $\HH_i(N,\beta)=0$ for all $i\geq0$,
\item $\HH_0(N,\beta)=0$,
\item $\beta\notin\qdeg(N).$
\qed
\end{enumerate}
\end{theorem}

\begin{theorem}{\cite[Theorem~6.6]{MMW},\cite{ekdi}}
\label{thm:higherHHvanishing}
Let $N$ be a weakly toric module.
Then $\HH_i(N,\beta)=0$ for all $i>0$ and for all $\beta\in\CC^d$
if and only if $N$ is a maximal Cohen--Macaulay $S_A$-module.
\qed
\end{theorem}

For a subset $\tau\subseteq A$, given an $\NN \tau$-module $\SS$, define the $S_\tau$-module $\CC\{\SS\}\defeq \bigoplus_{s\in\SS}\CC\cdot t^s$ as a $\CC$-vector space with $S_{\tau}$-action given by $\del_i\cdot t^s = t^{s+a_i}$. Then $\CC\{\SS\}$ has a multiplicative structure given by $t^s\cdot t^{s'} = t^{s+s'}$, and $S_\tau\cong\CC\{\NN \tau\}$ as rings.
The \emph{saturation} of $\tau$ in $\ZZ \tau$ is the semigroup
$\widetilde{\NN \tau}=\RR_{\geq 0}\tau\cap\ZZ \tau$.
The \emph{saturation} of $S_\tau$ is the semigroup ring 
of the saturation of $\tau$ in $\ZZ \tau$, which is given by $\widetilde{S}_\tau=\CC\{ \widetilde{\NN \tau} \}$ 
as a $\ZZ^d$-graded $S_\tau$-module.
By \cite{Hochster}, $\widetilde{S}_\tau$ is a Cohen--Macaulay $S_\tau$-module.

\section{Characteristic cycles of $A$-hypergeometric systems}
\label{sec:charC}

Let $L=(L_x,L_\del) \in \QQ^{2n}$ be a projective weight vector on $D$. 
In this section, we recall from~\cite{slopes} the description of the $L$-characteristic variety of an $A$-hypergeometric system, which includes the computation of the $L$-characteristic cycle of $\HA$ when $\beta$ is not rank--jumping for $A$. 

Let $h = (h_1,\dots,h_d)\in \QQ^d$ be such that $h\cdot a_i >0$ for $i=1,\dots,n$. 
Choose $\varepsilon>0$ such that $h\cdot a_i +\varepsilon\, L_{\del_i} >
0$ for $i=1,\dots,n$, and denote by $H_\varepsilon$ the hyperplane in $\PP^d_{\QQ}$
given by 
\[
\{ [y_0:y_1:\cdots:y_d] \in \PP^d_\QQ \mid \varepsilon \, y_0
+ h_1 y_1+\cdots +h_d y_d = 0\}.
\] 
The \emph{$L$-polyhedron of $A$} is the convex hull of $\{[1:\boldzero_d],
[L_{\del_1}:a_1],\dots,[L_{\del_n}:a_n]\}$ in the affine space $\PP^{d}_\QQ \minus H_\varepsilon$.
The \emph{$(A,L)$-umbrella}, denoted $\Phi_A^L$, is the set of
faces of the $L$-polyhedron of $A$ that do not contain $[1:\boldzero_d]$.  

We denote by $\Phi_A^{L,k}\subset \Phi_A^L$ the subset of faces $\tau$ of dimension $k$ (equivalently, $\dim ( \CC \tau)=k+1$). A face $\tau$ of $\Phi_A^L$ will be identified with $\{j\in\{1,\ldots ,n\} \mid [L_{\del_j}:a_j]\in \tau \}$ or with the submatrix of $A$ indexed by this set, when necessary. With this identification, $\Phi_A^L$ is an abstract polyhedral complex.
For any face $G\preceq A$, set 
$\Phi_G^L \defeq \{ \tau \in \Phi_A^L \mid \tau \subseteq G\}$. 

Let $(x,\xi)$ denote the coordinates on $T^*X = T^*\CC^n$. 
For any $\tau \subseteq \{1,\ldots ,n\}$, let 
\[
C_A^{\tau} 
  \defeq \left\{(x,\xi )\in T^*X 
  \,\big\vert\, 
  \xi_i =0 \mbox{ for } i \notin \tau, \,  \textstyle\sum_{i\in \tau } a_i x_i \xi_i 
  = 0 \text{ and } \exists t\in (\CC^*)^d, \,
  \xi_j = t^{a_j}, \, \forall j \in \tau  
 \right\},
\] 
and let $\overline{C_A^{\tau}}$ denote the Zariski closure of $C_A^{\tau}$ in $T^*X$, with defining ideal $P_{\tau}\subseteq \CC [ x ,\xi ]$. In particular, 
$\overline{C_A^{\nothing}}=T_X^*X$ and $\overline{C_A^{\{j\}}}=T_{(x_j =0)}^*X$.

If $N$ is a $\ZZ^d$-graded $\CC[\del]$-module and $C=\overline{C_A^{\tau}}$ for some $\tau \in \Phi_A^L$,  we write  
\begin{align}\label{eqn:muDef}
\mu_{A,i}^{L,\tau}(N,\beta)
  \defeq \mu^{L,C}(\HH_i (N,\beta))=\ell ( (\gr^L (\HH_i (N,\beta)))_{P_\tau}). 
\end{align}

We will also denote $\mu_{A,i}^{L,\tau}(\beta )\defeq \mu_{A,i}^{L,\tau}(S_A,\beta)$.
By \cite[Corollary 4.13]{slopes},  
\begin{align}\label{eqn:genericmuDef}\mu_A^{L,\tau}\defeq \sum_{j=0}^d (-1)^j  \mu_{A,i}^{L,\tau}(\beta)=
\mu_{A,0}^{L, \tau}(\widetilde{S}_A ,\beta )=\mu_{A,0}^{L, \tau}(S_A [\del_A^{-1}] ,\beta )
 \end{align} is independent of $\beta\in \CC^d$.

Note that $\rank(M_A(\beta))$ is equal to 
$\mu_{A,0}^{F,\varnothing}(\beta)$. Since $M_A(\beta)$ is always holonomic~\cite{GGZ,adolphson}, its rank is always finite. 
Further, the rank of $M_A(\beta)$ is upper semicontinuous as a function of the parameter $\beta$, with a generic value equal to $\vol_{\ZZ^d}(A)$, the normalized volume in $\ZZ A=\ZZ^d$ of the convex hull of the columns of $A$ and the origin~\cite{MMW, adolphson,GKZ-int}. We recall that the normalized volume function in a lattice $\Omega$, denoted by $\vol_{\Omega}$, is defined so that the volume of the unit simplex in $\Omega$ (that is, the convex hull of the origin and a lattice basis of $\Omega$) is one. 

A parameter $\beta$ is said to be \emph{rank-jumping} when $\rank(M_A(\beta))>\vol_{\ZZ^d}(A)$. 
The set of rank--jumping parameters is described in \cite{MMW}; namely, with $\varepsilon_A \defeq \sum_{i=1}^n a_i$,  
\[
\cE_A 
\defeq \{\beta\in\CC^d\mid \rank(M_A(\beta))>\vol_{\ZZ^d}(A)\} 
= -\, \qdeg\left( \bigoplus_{i=0}^{d-1} \Ext_{\CC[\del]}^{n-i}(S_A,\CC[\del])( - \varepsilon_A)\right).
\]
Schulze and Walther provided a description of $\charC^L(\HA)$ when $\beta$ is not rank--jumping, as summarized through the following two results.

\begin{theorem}\cite[Theorem~4.21]{slopes}\label{thm:swFormula}
For all $G\preceq A$, 
if $\tau\in\Phi^L_G$, then
\[
\mu^{L,\tau}_G 
= \sum_{\tau\subseteq\tau'\in\Phi^{L,d'-1}_G}
    [\ZZ G :\ZZ\tau'] \cdot [(\ZZ\tau'\cap\QQ\tau):\ZZ\tau]
    \cdot \vol_{\pi (\ZZ \tau ')}(P_{\tau,\tau'}\setminus Q_{\tau,\tau'}), 
\] 
where $d'=\dim(\CC G)$, $\pi \colon \ZZ \tau ' \twoheadrightarrow \ZZ \tau ' /(\ZZ \tau ' \cap \QQ \tau )$ is the natural projection and $P_{\tau , \tau '}$ and $Q_{\tau ,\tau '}$ denote the convex hull of $\pi (\tau ' \cup \{0\}) $ and $\pi (\tau ' \setminus \tau)$ respectively.
\end{theorem}

\cite[Theorem~4.21]{slopes} is only stated for $G=A$. Theorem \ref{thm:swFormula} is a straightforward adaptation that will be useful in the sequel. Note that here we are using \eqref{eqn:muDef} and \eqref{eqn:genericmuDef} with $A$ replaced by $G$, but we still write $L$ for the filtration induced on the Weyl Algebra $D_G$ in the variables $\{x_j\mid j\in G\}$ by the projective weight vector given by the $G$-coordinates of $L_x$ and $L_\del$.

\begin{theorem}\cite[Corollary~4.12]{slopes}
\label{thm:sw char}
The $L$-characteristic variety of $M_A (\beta)$ is independent of $\beta \in \CC^d$ and given by 
\[ \charVar^L (M_A (\beta )) 
  = \bigcup_{\tau \in \Phi_A^L} 
    \overline{C_A^{\tau}},
\] 
where each component $\overline{C_A^{\tau}}$ is irreducible. Moreover, $\mu_{A,0}^{L,\tau}(\beta )\geq \mu_{A}^{L,\tau}$, and equality holds if $\beta$ is not rank--jumping.
\end{theorem}

Theorem~\ref{thm:sw char} implies that when $\beta$ is not rank-jumping, 
\[
\charC^L(M_A(\beta)) 
  = \sum_{\tau\in\Phi_A^L} \mu_A^{L,\tau} 
    \cdot \overline{C_A^{\tau}}, 
\]
and for each $\tau\in\Phi_A^L$, the multiplicity $\mu_A^{L,\tau}$ is computed in Theorem~\ref{thm:swFormula}.

A subset $\tau\subseteq A$ is called $F$-\emph{homogeneous} if the set of columns of $A$ indexed by $\tau$ lie in a common affine hyperplane off the origin.
For a subset $\tau\subseteq A$, let $\Delta_{\tau}=\conv(\tau\cup \{\boldzero\})\subseteq \RR^d$ denote the convex hull of the origin and all the columns of $\tau$. 

By \cite[Corollary 4.22 and Remark 4.23]{slopes}, 
\begin{equation}
\mu^{L,\emptyset}_A = \vol_{\ZZ^d}\left(\bigcup_{\tau'\in\Phi_A^{L,d-1}} \Delta_{\tau'}\setminus \conv(\tau')\right).
\label{eqn:non-F-homogeneous-mult} 
\end{equation}
Hence if all the facets of the $(A,L)$-umbrella are $F$-homogeneous, then 
\begin{equation}
\mu^{L,\emptyset}_A = \vol_{\ZZ^d}\left(\bigcup_{\tau'\in\Phi_A^{L,d-1}} \Delta_{\tau'}\right).
\label{eqn:F-homogeneous-mult} 
\end{equation}

\section{$F$-characteristic cycles of initial ideals are $L$-characteristic cycles}
Given any real vector $w\in\RR^n$ and any left ideal $J\subseteq D$, we can consider the initial ideal $\ini_{(-w,w)}(J)$ as defined in \cite{SST}. We recall that by \cite[Theorem 2.2.1]{SST}, 
if $M=D/J$ is a holonomic $D$-module, then so is  $\gr^{(-w,w)}(M)\defeq D/\ini_{(-w,w)}(J)$ and, moreover, 
\begin{equation}
\rank (\gr^{(-w,w)}M)\leq \rank (M).\label{eqn:wrank-leq-rank}
\end{equation}

On the other hand, by \cite[Lemma 2.1.6]{SST}), for any weight vector $(u,v)\in\RR^n$ and $L=(-w,w)+\epsilon (u,v)$ with $\epsilon>0$ small enough,
\begin{equation}\label{eq:gr=}
\gr^{(u,v)}(\gr^{(-w,w)}(M))=\gr^L(M). 
\end{equation}

\begin{lemma}\label{characteristic-cycles}
If $M=D/J$ is a holonomic $D$-module, then 
for $L$ chosen as in \eqref{eq:gr=} with $(u,v)=F$,  
\[
\charC^F (\gr^{(-w,w)}(M))=\charC^L (M).
\] 
\end{lemma}

The holonomic rank of $\ini_{(-w,w)}(M_A (\beta))$, a central object of study in \cite{SST}, equals the multiplicity $\mu^{L,\nothing}_{A,0}(\beta)$ for $L=(-w,w)+\epsilon F$ and $\epsilon>0$ small enough. Notice that, by the form of $L$, all the facets of $\Phi_A^L$ are $F$-homogeneous. We will see in \S\ref{sec:mults} that for any projective weight vector $L$, the multiplicity $\mu^{L,\nothing}_{A,0}(\beta)$ equals the rank of a Gr\"obner deformation of $M_A(\beta)$ (see Corollaries~\ref{cor:hypergeometric-converse} and~\ref{cor:non-F-homogeneous-L}).

\section{Computing multiplicities in $L$-characteristic cycles}
\label{sec:mults}

In this section, we use the approach of~\cite{berkesch} to compute the multiplicities in the $L$-characteristic cycles of Euler--Koszul homology modules of the toric ring $S_A$. 
We first recall some definitions  from~\cite{berkesch,BFM-parametric}. 

For a face $G\preceq A$, consider the union of the lattice translates 
\begin{align}\label{ranking-face-lattice}
\EE_G^\beta\defeq 
\big[\ZZ^d\cap(\beta+\CC G) \big]\minus(\NN A+\ZZ G) = \bigsqcup_{b\in B_G^\beta} (b+\ZZ G),
\end{align}
where $B_G^\beta$ is a set of lattice translate representatives. 
As such, $|B^\beta_{G}|$ is the number of translates of $\ZZ G$ appearing in $\bbE^\beta$, which is by definition equal to the difference between $[\ZZ^d\cap \QQ G:\ZZ G]$ and the number of translates of $\ZZ G$ along $\beta+\CC G$ that are contained in $\NN A + \ZZ G$.

For a face $\tau\in\Phi_A^L$ of the $(A,L)$-umbrella, let $\EE_\tau^\beta$ denote the union of the ranking lattices $\EE_G^\beta$, where $G\preceq A$ contains $\tau$. 

\begin{theorem}\label{thm:computeMults}
Let $L$ be a projective weight vector and $\tau\in\Phi_A^L$ be a face of the $(A,L)$-umbrella. 
For each $i$ and $\beta$, 
the multiplicity $\mu_{A,i}^{L,\tau}(\beta)$, 
which is the coefficient of $\ol{C}_A^\tau$ in the characteristic cycle $\charC^L(\HH_i(S_A,\beta))$ \emph{(}see~\eqref{eqn:muDef}\emph{)}, 
can be computed from the combinatorics of the ranking lattices at $\beta$ and the $(A,L)$-umbrella $\Phi_A^L$. 
More precisely, there is a spectral sequence involving the faces of $\Phi_A^L$ that contain $\tau$ and the ranking lattices in 
$\EE_\tau^\beta$, 
from which $\mu^{L,\tau}_{A,i}(\beta)$ can be computed.
\end{theorem}

Before proving Theorem \ref{thm:computeMults}, we state some consequences.

\begin{cor}\label{cor:L-CC-and-umbrella}
For all $\beta\in\CC^d$ and all projective weight vectors $L,L'$,
\[\charC^L (M_A(\beta))=\charC^{L'} (M_A(\beta)) \quad\text{if and only if}\quad  \Phi_A^L =\Phi_A^{L'}.
\]
\end{cor}
\begin{proof}
While the only if direction follows from Theorem \ref{thm:sw char}, the if direction uses 
Theorems \ref{thm:swFormula}, \ref{thm:sw char}, and \ref{thm:computeMults}.
\end{proof}

\begin{cor}\label{cor:hypergeometric-converse}
For any projective weight vector $L=(u,v)$ on $D$ such that all the facets of $\Phi_A^L$ are $F$-homogeneous, 
\[
\charC^F(\gr^{(-v,v)}(M_A (\beta)))=\charC^L(M_A (\beta)).
\] In particular, $\rank(\gr^{(-v,v)}(M_A (\beta)))=\mu_{A,0}^{L,\nothing}(\beta)$.
\end{cor}
\begin{proof}
Let $\epsilon>0$ be as small as necessary in the sequel. Notice first that $\charC^F(\gr^{(-v,v)}(M_A (\beta)))=\charC^{L_\epsilon}(M_A (\beta))$ for $L_\epsilon\defeq (-v,v)+ \epsilon F$ by Lemma \ref{characteristic-cycles}.
Moreover, by the assumption on the $(A,L)$-umbrella, we have $\Phi_A^L=\Phi_A^{L+\epsilon F}$. On the other hand, the last $n$ coordinates of $L+\epsilon F$ and $L_\epsilon$ are equal to $v+\epsilon \cdot\boldone_n$, and hence $\Phi_A^{L}=\Phi_A^{L\epsilon}$. Thus, the result follows from Corollary \ref{cor:L-CC-and-umbrella}.
\end{proof}

As a particular case of Corollary \ref{cor:hypergeometric-converse}, the characteristic cycles, and hence the ranks, of the modules $\gr^{(-\boldone_n,\boldone_n)}(M_A (\beta))$ and $M_A(\beta)$ are equal.
We next show that~\cite[Corollary 3.2.14]{SST} holds with weakened hypotheses. 

\begin{cor}\label{cor:smallGFan}
For any $\beta\in\CC^d$ and any (not necessarily homogeneous) $A$, the small Gr\"obner fan of the hypergeometric ideal $H_A (\beta)$ refines the secondary fan of $A$. 
\end{cor}
\begin{proof}
It suffices to see that each open cone of the small Gr\"obner fan of $H_A (\beta)$ is contained in an open cone of the secondary fan of $A$. Since such an open cone corresponds to a Gr\"obner deformation with respect to a  generic weight vector $w\in\RR^n$, it follows that $L=(-w+c \cdot \boldone_n,w)$ is a projective weight vector for any $c>0$ and $\Phi_A^L$, which only depends on $w$, has only $F$-homogeneous facets. Thus, beginning with generic vectors $w,w'$ with 
\[
\gr^{(-w,w)}(M_A (\beta))=\gr^{(-w',w')}(M_A (\beta)),
\]  
Corollaries~\ref{cor:L-CC-and-umbrella} and~\ref{cor:hypergeometric-converse} imply that $\Phi_A^L =\Phi_A^{L'}$ where the last coordinates of $L$ and $L'$ are $w$ and $w'$ respectively. This means that $w$ and $w'$ belong to the same cone of the secondary fan of $A$.
\end{proof}

\begin{cor}\label{cor:non-F-homogeneous-L}
Any projective weight vector $L=(u,v)$ on $D$ has a perturbation $L'$ such that all the facets of the $(A,L')$-umbrella $\Phi_A^{L'}$ are $F$-homogeneous and $\mu_{A,0}^{L,\nothing}(\beta)=\mu_{A,0}^{L',\nothing}(\beta)$. 
\end{cor} 
\begin{proof}
If $L'(\epsilon)\defeq L+\epsilon (\boldone_n,-\boldone_n)$ for $\epsilon>0$, then there is an $\epsilon_0>0$ such that the $L'(\epsilon)$-umbrella is constant for $\epsilon \in (0,\epsilon_0]$. Thus, if we fix $L'=L'(\epsilon_0)$, then all the facets of $\Phi_A^{L'}$ are $F$-homogeneous. Moreover, by the choice of $L'$, any $F$-homogeneous facet of $\Phi_A^L$ is a facet of $\Phi_A^{L'}$, while each non-$F$-homogeneous facet $\tau$ of $\Phi_A^L$ is replaced in $\Phi_A^{L'}$ 
by the set of facets of $\Phi_{\tau}^{L''}$, where $L''\defeq (c\boldone_n,-\boldone_n)$ is a projective weight vector for any $c>1$. This latter set is the set of facets of $\conv(\tau)$ that are not facets of $\Delta_\tau$. This proves that $\mu_{A}^{L,\nothing}=\mu_{A}^{L',\nothing}$ by using~\eqref{eqn:non-F-homogeneous-mult} to compute $\mu_A^{L,\nothing}$ and \eqref{eqn:F-homogeneous-mult} to compute $\mu_A^{L',\nothing}$. Analogously, $\mu_{G}^{L,\nothing}=\mu_{G}^{L',\nothing}$ for any face $G\preceq A$. Finally, the result follows from previous equality and Theorem \ref{thm:computeMults}.
\end{proof}

\begin{cor}\label{cor:upperBound}
Given any projective weight vector $L$ and $\beta\in\CC^d$,
\[
\mu_{A,0}^{L,\nothing}(\beta)\leq \rank (M_A (\beta))\leq 4^{(d+1)}\vol (A).
\]
\end{cor}
\begin{proof}
The first inequality is a consequence of~\eqref{eqn:wrank-leq-rank} and Corollaries~\ref{cor:non-F-homogeneous-L} and~\ref{cor:hypergeometric-converse}. 
The second is~\cite[Corollary 6.2]{BFM-parametric}. 
\end{proof}

To prove Theorem~\ref{thm:computeMults}, we will follow the approach used to compute the rank of an $A$-hypergeometric system from~\cite{berkesch} (see also~\cite{BFM-parametric}). 
We will use the set
\[ \phantom{xxx}
  \cC_A(\beta) \defeq
\ZZ^d\cap (\Re\,\beta +\RR_{\geq 0}A). 
\]
Given a subset 
\begin{align}\label{eqn:J}
J\subseteq \mathcal{J}(\beta)\defeq\{(G,b)\mid G\preceq A,\, b\in B_G^\beta,\, \EE_G^\beta\neq\varnothing\}, 
\end{align}
define  
\begin{align*}
\EE_J^\beta 
  \defeq \bigcup_{(G,b)\in J} (b+\ZZ G)
  \qquad\text{and}\qquad 
\PP_J^{\beta} 
  \defeq \cC_A(\beta) \cap \EE_J^{\beta}. 
\end{align*}
Now define the respective sets and $S_A$-modules 
\begin{align*}
\TT^{\beta} 
  \defeq \NN A\cup\left[
  \cC_A(\beta)\cap 
  \bigcup_{\substack{b\in\PP^\beta}}(b+\widetilde{\NN A})\right]&,
  \quad 
  &T^{\beta} &\defeq \CC\{\TT^{\beta}\},\\
  \SS_J^{\beta}  
    \defeq \TT^{\beta}\setminus\PP_J^{\beta}&,\quad
  &S_J^{\beta} 
	&\defeq \CC\{\SS_J^{\beta}\},
	\qquad \text{and}\qquad
  P_J^{\beta} \defeq  \dfrac{T^{\beta}}{S_J^{\beta}}.
\end{align*}
The degree set of $P_J^\beta$ is $\deg(P_J^{\beta})=\PP_J^{\beta}$. 
If a toric module $N$ is isomorphic to $P_J^{\beta}$ for some 
$J\subseteq \mathcal{J}(\beta)$ and $\beta$,
then we say that $N$ is a \emph{ranking toric module} determined by $J$.
A \emph{simple ranking toric module} is a module isomorphic to $P_{G,J}^\beta \defeq P_{J(G)}^\beta$, where $G\preceq A$ is a fixed face of $A$ such that $\EE_G^\beta\neq\varnothing$ and 
\[
J(G) \defeq \{(G,b)\in J\mid b\in B_G^\beta\}.
\] 
When $J=\mathcal{J}(\beta)$, we suppress it from the notation and write 
$P^\beta$ (and $P_G^\beta$) in place of $P_J^\beta$ (and $P_{G,J}^\beta$, respectively). If $(G,b)\in J$ and there is not any other pair $(G',b')\in J$ such that $b+\ZZ G \subsetneq b' +\ZZ G '$ we say that $(G,b)$ is a maximal pair in $J$. 
We denote by $\max(J)$ the set of all maximal pairs in $J$.

\begin{lemma}\label{lem:simpleRT}
If $G\preceq A$ and $\tau\in\Phi^L_A$, then the multiplicity $\mu^{L,\tau}_{A,q}(P^\beta_G,\beta)$ of the simple ranking toric module $P^\beta_G$ is
\[
\mu^{L,\tau}_{A,q}(P^\beta_G,\beta) = |B^\beta_{G}|\cdot
\mu^{L,\tau}_{A,q}(\widetilde{S}_G (b),\beta) =
\begin{cases}
|B^\beta_{G}|\cdot \binom{\codim(G)}{q} \cdot \mu^{L,\tau}_G &\text{if $\tau \subseteq G$},\\
0 & \text{otherwise}.
\end{cases}
\] for any $b\in B^\beta_{G}$.
\end{lemma}
\begin{proof}
For all $j\notin G$ we have that $\del_j\cdot \widetilde{S}_G=0$, hence that $\xi_j \cdot \gr^L(\HH_0 ( \widetilde{S}_G (b),\beta))=0$ where $\xi_j=\ini_L(\del_j)\in \CC [x,\xi]\cong \gr^L (D)$. 
On the other hand, by the definition of $P_{\tau}$, it is clear that $\xi_j \in P_{\tau}$ if and only if $j\notin \tau$.
Thus, we have that $(\gr^L(\HH_0 ( \widetilde{S}_G (b),\beta)))_{P_{\tau}}=0$ if $\tau\nsubseteq G$. Now, with $\mu^{L,\overline{C_A^\tau}}$ in place of rank, 
the arguments in the proof of \cite[Theorem~6.1]{berkesch} yield this result.
\end{proof}

\begin{proof}[Proof of Theorem~\ref{thm:computeMults}]
The argument proving \cite[Theorem~6.6]{berkesch} can be used to obtain this result, when $J$ is chosen to be the right hand side of~\eqref{eqn:J} and $\mu^{L,\overline{C_A^\tau}}$ 
in place of rank. 
We make note of the necessary modifications below. 

To begin, it follows from Theorem~\ref{thm:higherHHvanishing} and \eqref{eqn:genericmuDef} that
\begin{equation}
\label{mu-Q}
\mu_{A,i}^{L,\tau}(\beta) = 
\begin{cases}\mu_A^{L,\tau} + \mu_{A,1}^{L,\tau}(Q_A,\beta) - \mu_{A,0}^{L,\tau}(Q_A,\beta) 
&\text{if $i=0$,}\\
\mu_{A,i+1}^{L,\tau}(Q_A,\beta)
&\text{if $i>0$,}
\end{cases}
\end{equation}
where $Q_A$ sits in the short exact sequence $0\to S_A\to S_A[\del^{-1}]\to Q_A\to 0$. 
Then \cite[Proposition 5.10]{berkesch} implies that 
\begin{equation}
\label{mu-P-Q}
\mu_{A,i}^{L,\tau}(Q_A,\beta) = \mu_{A,i}^{L,\tau}(P_J^\beta,\beta),
\end{equation}
where $J$ is equal to the right hand side of~\eqref{eqn:J}. 
Now \cite[Lemmas 6.9, 6.10, 6.11, and 6.14]{berkesch} can be applied verbatim, 
while Lemma~\ref{lem:simpleRT} replaces the need for \cite[Lemma~6.13]{berkesch}. 
Finally, as \cite[Lemmas 6.12 and 6.15]{berkesch} hold when rank is replaced with $\mu_A^{L,\overline{C_A^\tau}}$, which is possible since localization at $P_\tau$ and $\gr^L(-)$ are exact functors and length is additive, the arguments of the proof of \cite[Theorem~6.6]{berkesch} yield the desired result. 
In particular, the spectral sequence involved begins with the cellular resolution of $P_J^\beta$ as constructed in \cite[(6.3)]{berkesch}:
\begin{equation}\label{eq:cellularComplex}
0 \rightarrow P_J \rightarrow I_J^0 \rightarrow I_J^1 \rightarrow \cdots \rightarrow I_J^{r}\rightarrow 0,
\end{equation}
where $I_J^\bullet$ is constructed as follows. 
Set 
\begin{eqnarray*}
\Delta^0_J &=& \{ F\preceq A \mid \exists (F,b)\in \max(J) \}, \\ 
\Delta^p_J &=& \{ s \subseteq \Delta^0_J \mid |s| = p+1\}, \text{ and} \\
F_s &=& \bigcap_{G\in s} G \ \ \text{for $s\in\Delta^p_J$}. 
\end{eqnarray*} 
With $r +1 = |\Delta^0_J|$, 
let $\Delta = \Delta_J^{\beta}$ be the standard $r$-simplex  
with vertices corresponding to the elements of $\Delta^0_J$. 
To the $p$-face of $\Delta$ spanned by the vertices 
corresponding to the elements in $s\in\Delta^p_J$, 
assign the ranking toric module $P_{F_s,J}^{\beta}$. 
Choosing the natural maps 
$P_{F_s,J}^{\beta} \rightarrow P_{F_t,J}^{\beta}$ 
for $s\subseteq t$ 
induces a cellular complex supported on $\Delta$, 
\begin{eqnarray}\label{9}
    I_J^\bullet\colon\quad 
    I_J^0 \rightarrow
    I_J^1 \rightarrow \cdots \rightarrow I_J^r \rightarrow 0
    \quad\text{with}\quad 
I_J^p = \bigoplus_{s\in\Delta^p_J} P_{F_s,J}^{\beta}.    
\end{eqnarray}
Applying Euler--Koszul homology to~\eqref{9} yields a double complex. The desired spectral sequence arises from this double complex after localizing at $P_\tau$ and applying $\gr^L(-)$. 
\end{proof}

\begin{remark}\label{ex:2cmpts}
If $\beta\in\CC^d$ is such that 
$\max(\mathcal{J}(\beta))$ involves two faces, $F_1, F_2$, 
then the proof of Theorem~\ref{thm:computeMults} 
shows that 
\begin{align}\label{eq:2cmpts}
\mu_{A,0}^{L,\tau}(\beta) - \mu_A^{L,\tau} = 
    \sum_{i=1}^2 \left(
    |B_{F_i}^{\beta}| \cdot [\codim(F_i)-1] \cdot \mu_{F_i}^{L,\tau} 
    \right) + |B_G^{\beta}| \cdot C^{\beta} \cdot\mu_{G}^{L,\tau},
\end{align}
where 
$G = F_1\cap F_2$ and 
the constant $C^{\beta}$ is given by
\begin{eqnarray*}
C^{\beta} = \binom{\codim(G)}{2} - \codim(G) + 1 
    - \binom{\codim(F_1)}{2} - \binom{\codim(F_2)}{2} +
    \binom{\codim(\CC F_1+\CC F_2)}{2}. 
\end{eqnarray*}
\end{remark}

\begin{example}\label{notsimplicial} 
The values of the $\mu_{A,0}^{L,\tau}(\beta)$ for a fixed $\beta$ are dependent upon the choice of face $\tau\in\Phi_A^L$. 
For example, consider the matrix 
\[
A = \begin{bmatrix}
     2 & 3 & 1 & 0 & 0 & 0 & 1\\
     0 & 0 & 0 & 2 & 3 & 1 & 1\\
     0 & 0 & 1 & 0 & 0 & 1 & 0\\
    \end{bmatrix},
    \]
    and the parameter $\beta=(0,0,-1)^t$, which lies outside the cone $\RR_{\geq 0} A$. 
    It turns out that 
\[
\widetilde{\NN A}\setminus\NN A=(\beta + \NN G_1)\cap \RR_{\geq 0} A=(\beta + \NN G_1)\setminus\{\beta\},
\] 
where $G_1=\{a_3,a_6\}$ and $G_2=\{a_1,a_2,a_5,a_7\}$ are facets of $A$. 
    In particular, $\cE_A=\{\beta\}$ and the ranking lattices at $\beta$ are  
\[
\EE^\beta =(\beta + \ZZ G_1)\cup (\beta + \ZZ G_2).
\] 
By Remark \ref{ex:2cmpts}, $\mu^{L,\varnothing}_{A,0}(\beta)-\mu^{L,\varnothing}_A  =1$ for any projective weight vector $L$. On the other hand, $\mu^{L,\tau}_{A,0}(\beta)=\mu^{L,\tau}_A$ if $\tau\neq\nothing$.
\end{example}

\begin{example}\label{ex:Lmatters-d=3}
The choice of projective weight vector impacts the resulting stratification via multiplicities of $\cE_A$. 
For example, consider the matrix 
\[
A = \begin{bmatrix}
2 & 3 & 0 & 0 & 0 & 0 & 1\\
0 & 0 & 1 & 3 & 0 & 0 & 1\\
0 & 0 & 0 & 0 & 1 & 2 & 1
\end{bmatrix},
\]
which has 
\[
\widetilde{\NN A}\setminus \NN A=(\beta+\NN G_1)\cup (\beta+\NN G_2),
\] 
where $G_1=\{a_3,a_4\}, G_2=\{a_5,a_6\}\preceq A$ and $\beta = (1,0,0)^t$. Moreover, we also have 
\[
\EE^\beta = (\beta+\ZZ G_1)\cup (\beta+\ZZ G_2)
\quad\text{and}\quad 
\cE_A=(\beta+\CC G_1)\cup(\beta+\CC G_2).
\]
If $L_\del = (1,4,1,4,1,3,1)$ and $L_x = 5\cdot\boldone_7-L_\del$, then $\mu^{L,\varnothing}_{A,0}(\beta')-\mu^{L,\varnothing}_A  =1$ for any $\beta' \in \cE_A$. On the other hand, the stratification of $\cE_A$ by the rank jump is different: 
\begin{align*}
\mu^{F,\varnothing}_{A,0}(\beta')-\mu^{F,\varnothing}_A= \begin{cases}
2 &\text{if $\beta' \in(\beta+\CC G_2)\setminus\{\beta\}$}\\
3 &\text{if $\beta' \in(\beta+\CC G_1)\setminus\{\beta\}$}\\
4 &\text{if $\beta' =\beta$}.
\end{cases}
\end{align*}
\end{example}

\section{More consequences of the multiplicity computation}
\label{sec:consequences}

For $\tau \in \Phi_A^L$, let 
\[
j_A^{L,\tau}(\beta) 
\defeq \mu^{L,\tau}_{A,0}(\beta)-\mu^{L,\tau}_A
\] 
be the \emph{$(L,\tau)$-multiplicity jump} at $\beta$, and let 
\[
\cE_A^{L,\tau}
  \defeq \{ \beta \in \CC^d \mid 
  j_A^{L,\tau}(\beta)>0\}
\] 
be the \emph{$(L,\tau)$-exceptional set} of $A$.
In this section, we record consequences of Theorem~\ref{thm:computeMults} and its implications for $\cE_A^{L,\tau}$. 
We also propose a description of $\cE_A^{L,\tau}$ and prove it holds in a special case. 

\begin{cor}\label{cor:not-in-any-codim2}
If $\tau\in \Phi_A^L$ is a face of the $(A,L)$-umbrella such that $\tau$ is not contained in any face of $A$ of codimension $2$, then $\cE_A^{L,\tau} = \nothing$. 
\end{cor}

\begin{proof}
Fix $\beta\in\CC^d$. 
By hypothesis, $\tau$ is contained in at most one facet of $A$. 
Recall that the cellular resolution of $P_J^{\beta}$ is made of ranking toric modules $P_G^{\beta}$ for faces $G\preceq A$ such that $\bbE_G^{\beta}\neq \nothing$.
 
If $\tau$ is not contained in any proper face of $A$
or it is contained in a unique facet $F\preceq A$ with $\bbE_F^{\beta}=\nothing$, then Lemma~\ref{lem:simpleRT} guarantees that $\mu^{L,\tau}_{A,q}(P^\beta_G,\beta)=0$ for all $q\geq 0$ for any proper face $G\preceq A$ with $\EE_G^{\beta}\neq\nothing$. 
Thus, the formula from Theorem~\ref{thm:computeMults} computes that 
$\mu_{A,i}^{L,\tau}(P^{\beta},\beta)=0$ for all $i\geq 0$. 

For the remaining case when $\tau$ is contained in a unique facet $F\preceq A$ and $\bbE_F^{\beta} \neq \nothing$,  
\[
\mu_{A,i}^{L,\tau}(P^{\beta},\beta)
= \mu_{A,i}^{L,\tau}(P_F^{\beta} , \beta ) 
= |B^\beta_{F}|\cdot
  \mu^{L,\tau}_{A,i}(\widetilde{S}_F,\beta) 
=|B^\beta_{F}|\cdot \binom{1}{i} \cdot 
  \mu^{L,\tau}_F
\] 
for all $i\geq 0$. 
Therefore, as in the proof of Theorem~\ref{thm:computeMults},
\[
j_A^{L,\tau}(\beta) 
= \mu_{A,1}^{L,\tau}(P^\beta ,\beta ) - 
  \mu_{A,0}^{L,\tau}(P^\beta ,\beta )
= 0.
\qedhere
\] 
\end{proof}

\begin{remark}
\label{cor:codim1mu}
As an immediate consequence of Corollary \ref{cor:not-in-any-codim2}, if $\dim(\CC\tau)\geq d-1$, then $\mu^{L,\tau}_{A,i}(\beta)$ is independent of $\beta$. 
Notice that this fact was known when $\dim(\CC\tau)=d$ (see \cite[Theorem 3.10]{slopes}).
\qed
\end{remark}

\begin{cor}
\label{cor:uniqueCodim2}
If $\tau\in \Phi_A^L$ is a face of the $(A,L)$-umbrella such that $\tau$ is contained in a unique face $G\preceq A$ of codimension $2$, then 
\[
j_A^{L,\tau}(\beta)=
\begin{cases}
|B_G^{\beta}|\cdot \mu_G^{L,\tau} & \text{ if $(G,b) \in \max(\mathcal{J}(\beta))$ for $b\in B_G^{\beta}$, (see \eqref{eqn:J}),} \\
0 & \text{ otherwise.}
\end{cases}
\]
\end{cor}
\begin{proof}
By the proof of Theorem~\ref{thm:computeMults} and Lemma~\ref{lem:simpleRT}, 
\[
j_A^{L,\tau}(\beta) 
= \mu_{A,1}^{L,\tau}(P^\beta ,\beta ) - 
  \mu_{A,0}^{L,\tau}(P^\beta ,\beta )
=  \mu_{A,1}^{L,\tau}(P_{J'}^\beta ,\beta ) - 
  \mu_{A,0}^{L,\tau}(P_{J'}^\beta ,\beta ),
\]
where $J'=\{(F,b)\in J \mid \tau \subseteq F\}$ for $J=\mathcal{J}(\beta)$. 
If $(G,b)\in \max(J)$ and $b\in B_G^{\beta}$, then $J' =J(G)$ and $P_{J'}^{\beta}=P_{G}^{\beta}$.
Thus, it is enough to consider the case when $(G,b)\notin \max(J)$ for any $b$ but there exists at least one facet $F$ such that $\tau \subseteq G\preceq F$ and $(F,b)\in \max(J)$. 
In this case, either $\max(J')=J(F)$ or $J'= J(F)\cup J(F')$ for some other facet $F'$ such that $F\cap F' = G$. 
Either way, it follows that  
$\mu_{A,1}^{L,\tau}(P_{J'}^\beta ,\beta ) - 
  \mu_{A,0}^{L,\tau}(P_{J'}^\beta ,\beta ) = 0$.
\end{proof}

By Corollaries \ref{cor:not-in-any-codim2} and \ref{cor:uniqueCodim2}, if $\dim(\CC\tau) = d-2$, then $j_A^{L,\tau}(\beta)>0$ only when there is a (unique) codimension $2$ face $G$ of $A$ containing $\tau$ and 
$(G,b)\in \max(\mathcal{J}(\beta))$ for some $b\in B_G^{\beta}$. 

\begin{notation}
For any $\tau \in \Phi_A^L$, let us denote $S_A^{\tau}\defeq \CC [(\NN A + \ZZ \tau )\cap \RR_{\geq 0} A]$.
\end{notation}

\begin{conjecture}
\label{conj:exceptional-Ltau}
There is an equality 
\[
\cE_A^{L,\tau}=-\,  \operatorname{qdeg}
  \left(
  \bigoplus_{q=0}^{d-1} \Ext^{n-q}_{\CC[\del]} (
  S_A^{\tau},\CC[\del])
  (-\varepsilon_A)\right),
\] 
where $\varepsilon_A \defeq \sum_{i=1}^n a_i$. In particular, $\cE_A^{L,\tau} =\nothing$ if and only if $S_A^{\tau}$ is Cohen--Macaulay.
\end{conjecture}

As evidence of the truth of Conjecture~\ref{conj:exceptional-Ltau}, we exhibit a containment between the two sets involved. We then prove the second part of conjecture in the case that $\RR_{\geq 0}A$ is a simplicial cone.  

\begin{prop}\label{prop:inclusion}
There is a containment 
\[
\cE_A^{L,\tau}
\, \subseteq \ \, 
-\,  \operatorname{qdeg}
  \left(
  \bigoplus_{q=0}^{d-1} \Ext^{n-q}_{\CC[\del]} (
  S_A^{\tau},\CC[\del])
  (-\varepsilon_A)\right).
\]
\end{prop}
\begin{proof}

By the definition of $S_A^{\tau}$, it is clear that $S_A^{\tau} [\del_\tau^{-1}]= S_A [\del_\tau^{-1}]$ and thus,
\[
\mu^{L,\tau}_{A,0}(S_A,\beta) 
= \mu^{L,\tau}_{A,0}(S_A [\del_\tau^{-1}],\beta)
= \mu^{L,\tau}_{A,0}(S_A^{\tau}[\del_\tau^{-1}], \beta) 
= \mu^{L,\tau}_{A,0}(S_A^{\tau}, \beta),
\] 
where the first and third equalities follows from the definition of $\mu^{L,\tau}_{A,0}$ (see \eqref{eqn:muDef}) and the fact that $\xi_j=\operatorname{in}_L(\del_j) \notin P_{\tau}$ if and only if $j\in\tau$.

If $\beta \notin -\operatorname{qdeg}
  (\Ext^{n-q}_{\CC[\del]} (S_A^{\tau},\CC[\del])(-\varepsilon_A)) $
for any $q=0, \ldots , d-1$, then 
$\HH_i(S_A^{\tau},\beta)=0$ for all $i>0$ by \cite[Theorem~6.6]{MMW}. 
Thus, $\mu_{A,0}^{L,\tau}(S_A^{\tau},\beta)=\sum_{j=0}^d (-1)^j \mu_{A,j}^{L,\tau}(S_A^{\tau},\beta)$, which is independent of $\beta$ by \cite[Theorem~4.11]{slopes} and hence equal to the generic value $\mu_A^{L,\tau}$.  
In particular, $\beta \notin \cE_A^{L,\tau}$.
\end{proof}

\begin{prop}
\label{prop:PJhigherVanishing-specialCase}
Fix $\beta\in\CC^d$ and let $J$ be as in~\eqref{eqn:J}.
If $J$ involves only facets of $A$ satisfying that the intersection of $r$ of them is a face of codimension at most $r$, then $\HH_q (P_J^\beta,\beta) = 0$ for all $q\geq 2$. 
\end{prop}
\begin{proof}
Consider the cellular resolution of $P_J^\beta$ as constructed in \cite[(6.3)]{berkesch}:
\[
0 \rightarrow P_J \rightarrow I_J^0 \rightarrow I_J^1 \rightarrow \cdots \rightarrow I_J^{r}\rightarrow 0,
\] 
where $r+1$ is the cardinality of $J$. 
On the other hand, if 
$K_p \defeq \ker ( I_J^{p}\rightarrow I_J^{p+1})$ for $0\leq p \leq r-1$ and $K_r = I_J^r$, then there are short exact sequences: 
\[
0 \rightarrow P_J \rightarrow I_J^0 \rightarrow K_1 \rightarrow 0
\qquad\text{and}\qquad
0 \rightarrow K_p \rightarrow I_J^p \rightarrow K_{p+1} \rightarrow 0 
\ \ \text{for} \ 1\leq p \leq r-1.
\]
By the assumption on $J$, $I_J^{p}$ is a direct sum of simple ranking toric modules $P_G$ for faces $G$ of codimension at most $p+1$, so by \cite[Proposition 3.2]{berkesch},  
$\HH_q (I_J^p,\beta )=0$ 
for all $q\geq p+2$ and $p=0,\dots,r$. 
Therefore 
\[
\HH_q (P_J , \beta ) 
\cong \HH_{q+1}(K_1 , \beta )
\cong \cdots 
\cong \HH_{q+r-1}(K_{r-1}, \beta)
\cong \HH_{q+r}(I_J^{r} ,\beta )=0
\] 
for all $q\geq2$, as desired. 
\end{proof}

Note that if $\RR_{\geq 0} A$ is simplicial, then any set of facets of $A$ satisfies the property required in Proposition~\ref{prop:PJhigherVanishing-specialCase}. To the contrary, Example \ref{notsimplicial} does not satisfy this property.

\begin{theorem}
\label{thm:invertCM}
Let $\tau\in\Phi^L_A$ and assume that $\RR_{\geq 0}A$ is a simplicial cone. Then $\cE_A^{L,\tau} =\nothing$ if and only if $S_A^{\tau}$ is Cohen--Macaulay. 
\end{theorem}

\begin{proof}
The if direction is proven in Proposition~\ref{prop:inclusion}. 
By the definition of $S_A^{\tau}$ we have that 
\[
\operatorname{rank}(\HH_0 (S_A^{\tau},\beta ))=\vol (A) + \mu_{A,1}^{F,\nothing}(P_{J'})-\mu_{A,0}^{F,\nothing}(P_{J'}),
\]
where $J'\defeq\{(G,b)\in \mathcal{J}(\beta) \mid \tau \subseteq G\}$.
If $S_A^{\tau}$ is not Cohen--Macaulay, then by Theorem~\ref{thm:higherHHvanishing}, there exists $\beta \in \CC^d$ such that $\operatorname{rank}(\HH_0 (S_A^{\tau},\beta ))>\vol(A)$. 
Since $\RR_{\geq 0}A$ is simplicial, by Proposition~\ref{prop:PJhigherVanishing-specialCase} there must be a face 
$G$ of codimension at least $2$ such that $(G,b)\in \max(J')$. 
Thus, for generic $\beta' \in b + \CC G$, we have that $\operatorname{max}(\mathcal{J}(\beta '))=\{(G,b_1), \ldots, (G, b_r)\}$ with $r=|B_G^{\beta '}|$. Now, using Lemma~\ref{lem:simpleRT}, 
we have that 
\[
\mu_A^{L,\tau}(S_A ,\beta ')
  = \mu_A^{L,\tau}(S_A^{\tau} ,\beta ')
  = \mu_A^{L,\tau} + r 
    (\operatorname{codim}(G)-1)
    \cdot\mu_G^{L,\tau}>\mu_A^{L,\tau}
\] 
and thus 
$\beta'\in\cE_A^{L,\tau}\neq \nothing$.
\end{proof}

\section{Upper-semicontinuity and convex filtrations}
\label{sec:upperSC}

It was conjectured in \cite{slopes} that the multiplicities 
$\mu_{A,0}^{L,\tau}(\beta)$ 
are upper semicontinuous in $\beta\in\CC^d$ for any projective $L$ and $\tau\in\Phi_A^L$. 
We prove this conjecture when $L$ and $\tau$ satisfy certain conditions with respect to $A$ (see Theorem \ref{thm:upper-semicont-convex-case} and Corollary \ref{cor: upper-semicont-tau-convex-case}). We also prove Conjecture \ref{conj:exceptional-Ltau} in this setting when $\tau=\nothing$ (see Corollary \ref{cor:EAconvex}).

Given a submatrix $\sigma\subseteq A$ with rank $d$, denote by $E_{i}^\sigma$ the Euler operator associated with the $i$-th row of the matrix $\sigma$.  
Let $D_\sigma$ denote the Weyl algebra associated to the variables $x_\sigma = \{x_i\mid a_i\in \sigma\}$. 
We have that $\ZZ A=\ZZ^d =\bigoplus_{j=1}^{r} \Lambda_j$, where $r=[\ZZ^d:\ZZ\sigma]$ and $\Lambda_j=b_j +\ZZ \sigma$ for some $b_j\in\ZZ^d$ with  $j=1,\ldots,r$.

If $N$ is a $\ZZ^d$-graded $S_A$-module, then $N_{j}\defeq\bigoplus_{\alpha\in \Lambda_j} N_{\alpha}$ is a $S_\sigma$-module. Let $\KK^\sigma_\bullet(N,\beta)$ denote the direct sum over $j$ of the Euler--Koszul complexes on  
$D_\sigma \otimes_{\CC[\del_\sigma]} N_j (-b_j)$ given by the operators $\{E_i^\sigma - \beta_i + (b_j)_i\}_{i=1}^d$, where each such Euler--Koszul complex is placed in degree $b_j$. That is, 
\[
\KK^{\sigma}(N,\beta)\defeq\bigoplus_{j=1}^r \KK^{\sigma}(N_j(-b_j),\beta-b_j)(b_j),\] 
where the right-hand side Euler--Koszul complexes where defined before since $N_j(-b_j)$ is a $\ZZ\sigma$-graded $S_\sigma$-module. This definition is independent of the chosen elements $b_1,\ldots,b_r\in\ZZ^d$ by \eqref{eqn:shiftHH}.
With this setup, $D_{\sigma}\otimes N\cong \bigoplus_{j=1}^r (D_{\sigma}\otimes N_j)$, and  $\KK^\sigma_\bullet(N,\beta)$ is a $\ZZ^d$-graded complex of left $D_{\sigma}$-modules.
Set 
\[
\HH^\sigma_i(N,\beta) \defeq H_i(\KK^\sigma_\bullet(N,\beta)),
\] 
and note that these definitions make \eqref{eqn:shiftHH} and Theorem \ref{thm:HHvanishing} also valid for the homology modules $\HH^\sigma_i(N,\beta)$.

Let $L$ be a projective weight vector that induces a filtration on $D$ as considered in the introduction. 
We denote by $A^L$ the submatrix of $A$ whose columns belong to facets of $\Phi_A^L$. 
We say that $L$ is a \emph{convex} filtration with respect to $A$ if all facets of $\Phi_A^L$ are $F$-homogeneous and 
\begin{equation}
\bigcup_{\tau' \in \Phi_A^{L,d-1}}
\Delta_{\tau'} \label{convex-filtration}
\end{equation}
is a convex polytope, and thus equal to $\Delta_{A^L}$. 
Notice that, by the inclusion $S_{A^L}\subseteq S_A$, the ring $S_A$ is an $S_{A^L}$-module.

\begin{theorem}
\label{thm:upper-semicont-convex-case}
If $L$ is a convex filtration with respect to $A$, then
$\mu^{L,\nothing}_{A ,0} (\beta )
  =\rank (\HH^{A^L}_{0}(S_A , \beta))$. 
In particular, $\mu^{L,\nothing}_{A,0}(\beta)$ is upper-semicontinuous in $\beta$.
\end{theorem}

Before proving Theorem~\ref{thm:upper-semicont-convex-case}, we first consider the simple case. 

\begin{prop}\label{prop:mu-versus-rank}
Let $L$ be a convex filtration of $D$ with respect to $A$ and $G\preceq A$.
Then for all $(G,b)\in \mathcal{J}(\beta)$,  
\[
\mu_G^{L,\nothing} 
  = \vol_{\ZZ G} (G^L)
  = \rank \left(
    \HH_{0}^{G^L}(P^{\beta}_{(G,b)},\beta)\right),
\] 
where $G^L$ denotes the submatrix of $A$ whose columns belong to facets of $\Phi_G^{L}$.
\end{prop}
\begin{proof} 
The first equality follows from Theorem~\ref{thm:swFormula} and \eqref{eqn:F-homogeneous-mult} since $L$ is convex. 
For the second equality, by definition of the $(G,L)$-umbrella $\Phi_G^L$, the submatrix $G^L$ of $G$ is such that
$\RR_{\geq 0}G = \RR_{\geq 0}G^L$ 
and $\rank(G)=\rank(G^L)$. 
This implies that $S_G$ is a toric $S_{G^L}$-module. 
Further, 
\[
 P^{\beta}_{(G,b)}  \subseteq S_G [\del_G^{-1}](b)
   \cong \bigoplus_{\alpha \in \Lambda}
   S_{G^L}[\del_{G^L}^{-1}](\alpha),
 \]
where $\Lambda$ is a finite subset of $b + \ZZ G$ of cardinality $[\ZZ G:\ZZ G^L]$. Since 
\[
\deg(S_G [\del_G^{-1}](b)/P^{\beta}_{(G,b)})=(b+\ZZ G)\setminus \PP^\beta_{(G,b)},
\] 
it follows from  the definition of $C_A (\beta)$ that the parameter $\beta$ does not belong to the quasidegrees set of the weakly toric module $S_G [\del_G^{-1}](b)/P^{\beta}_{(G,b)}$. Thus, since $S_{G^L}[\del_{G^L}^{-1}]$ is a Cohen--Macaulay $S_{G^L}$-module, by Theorem~\ref{thm:HHvanishing} and Theorem~\ref{thm:higherHHvanishing}, 
$\HH_{i}^{G^L}(P^{\beta}_{(G,b)},\beta)=0$ 
for all $i\geq 1$ and 
\[
\rank\left( \HH_{0}^{G^L}(P^{\beta}_{(G,b)},\beta) \right)
  = [\ZZ G:\ZZ G^L]\cdot\vol_{\ZZ G^L}(G^L)
  = \vol_{\ZZ G}(G^L).
\qedhere
\]
\end{proof}

\begin{remark}\label{remark-ranking-lattices}
Notice that any weakly toric $S_A$-module $M\subseteq S_A [\partial_A^{-1}]$ can be viewed as a weakly toric $S_{A^L}$-module. Indeed, since $A^L$ and $A$ have the same rank, then $\ZZ A = \bigoplus_{j=1}^r (b_j + \ZZ A^L )$ for some $b_j \in \ZZ A$ with $j=1, \ldots , r$. Thus   
$S_A [\partial_A^{-1}] =  \bigoplus_{j=1}^r  S_{A^L}[\partial_{A^L}^{-1}](b_j)$ as $S_{A^L}$-modules. Setting  
$M_j \defeq M \cap S_{A^L}[\partial_{A^L}^{-1}](b_j)$, then $M$ is the direct sum of the weakly toric $S_{A^L}$-modules $M_j$. 
Moreover, for any face 
$G \preccurlyeq A$, 
\begin{align}\label{ranking-lattice-decomposition}
\EE_G^\beta =\bigsqcup_{b\in B_G^\beta} (b+\ZZ G)= \bigsqcup_{c \in B_{G^L}^{\beta}}  (c +\ZZ G^L),
\end{align} 
where $B_G^{\beta}$ and $B_{G^L}^{\beta}$ is a set of lattice representatives (see (\ref{ranking-face-lattice})). 
\end{remark}

\begin{lemma}\label{lemma-mult-convex}
The module $P^{\beta}$ is a direct sum of toric $S_{A^L}$-modules, and for any face $G\preceq A$ and $q\geq 0$, 
\[
\mu_{A,q}^{L,\nothing}(P_G^{\beta},\beta )=\mu_{A^L, q}^{F,\nothing}(P_G^{\beta},\beta).
\]
\end{lemma}
\begin{proof}
The decomposition of $M=S_A$ as a direct sum of weakly toric $S_A^L$-modules $M_j$ given in Remark \ref{remark-ranking-lattices} induces a decomposition of 
$S_A [\del_A^{-1} ]/M$ as a direct sum of the weakly toric $S_{A}^L$-modules $S_{A^L}[\partial_{A^L}^{-1}](b_j)/ M_j$. Then, by the two short exact sequences in the proof of \cite[Proposition 5.10]{berkesch}, $P^{\beta}$ is a direct 
sum of weakly toric $S_{A^L}$-modules. Moreover, since $\PP^{\beta}=\EE^\beta \cap \cC_A(\beta)$ and $\cC_A(\beta)=\cC_{A^L}(\beta)$, it follows 
that $P^{\beta}$ is a direct sum of toric $S_{A}^L$-modules.

On the other hand, if $G$ is a face of $A$, then by (\ref{ranking-lattice-decomposition}), $|B_G^{\beta}| [\ZZ G : \ZZ G^L ]=  |B_{G^L}^{\beta}|$. Thus, using Lemma \ref{lem:simpleRT} and Proposition \ref{prop:mu-versus-rank},
\begin{align*}
\mu_{A,q}^{L,\nothing}(P_G^{\beta},\beta )
&= |B^\beta_{G}|\cdot \binom{\codim(G)}{q} \cdot \vol_{\ZZ G} (G^L ) \\
&= |B_{G^L}^{\beta}|\cdot \binom{\codim(G)}{q} \cdot \vol_{\ZZ G^L} ( G^L )\\ &=\mu_{A^L, q}^{F,\nothing}(P_G^{\beta},\beta).\qedhere
\end{align*}
\end{proof}

The proof of Theorem~\ref{thm:upper-semicont-convex-case} makes use of the notion of a \emph{holonomic family} from \cite[Definition~2.1]{MMW}, which we now recall. 
While defined over any algebraic variety $B$ with structure sheaf $\cO_B$, we will need only the case when $B = \AA_\CC^d$, affine $d$-space over $\CC$. 

If $\beta\in B$, denote by $p_\beta$ the prime ideal (sheaf) of $\beta$ and set $\kappa_\beta = \cO_{B,\beta}/p_\beta\cO_{B,\beta}$,  the residue field of the stalk $\cO_{B,\beta}$. 
A \emph{coherent sheaf} of $(D\otimes_\CC\cO_B)$-modules is a quasi-coherent sheaf of $\cO_B$-modules on $B$ whose sections over each open affine subset $U\subset B$ are finitely generated over the ring of global sections $H^0(B,D\otimes_\CC\cO_U)$.
Let $\cO_B(x)$ denote the localization at $\<0\>\in\Spec(\CC[x])$ of $\cO_B[x] \defeq \CC[x]\otimes_\CC\cO_B \subset D\otimes_\CC\cO_B$. 
The sheaf-spectrum of $\cO_B(x)$ is the base-extended scheme $B(x) \defeq \Spec\,\CC(x)\times_{\Spec\,\CC}B$. 

A \emph{holonomic family} over $B$  is a coherent sheaf $\widetilde{\MM}$ of left $(D\otimes_\CC\cO_B)$-modules such that 
\begin{enumerate}
\vspace{-2.75mm}
\item the fibers $\MM_\beta = \widetilde{\MM}\otimes_{\cO_B}\kappa_\beta$ are holonomic $D$-modules for all $\beta\in B$, and 
\item $\cO_B(x)\otimes_{\cO_B}\widetilde{\MM}$ is coherent on $B(x)$. 
\end{enumerate}

\begin{proof}[Proof of Theorem~\ref{thm:upper-semicont-convex-case}]
Since $S_A [\del^{-1}]$ is a maximal Cohen--Macaulay weakly toric $S_{A^L}$-module, 
$\HH_{i}^{A^L}(S_A [\del^{-1}] ,\beta)=0$ for all $i>0$ by Theorem~\ref{thm:higherHHvanishing}. 
Thus, applying Euler--Koszul homology with respect to $A^L$ to the short exact sequence
\[
0\to S_A\to S_A [\del^{-1}] \to Q\to 0
\] 
and using that $\HH_{q}^{A^L}(P^{\beta},\beta)\simeq \HH_{q}^{A^L}(Q,\beta)$ (see the proof of \cite[Proposition 5.10]{berkesch}, which can be adapted to this case), it follows that
\[
\rank\left( \HH_{0}^{A^L}(S_A,\beta)\right)
  = \rank\left(\HH_{0}^{A^L}(S_A [\del^{-1}],\beta)\right)
	+ \mu_{A^L,1}^{F,\nothing}(P^\beta,\beta) 
	- \mu_{A^L,0}^{F,\nothing}(P^\beta,\beta).
\] 
The proofs of \cite[Theorem 6.6]{berkesch} and Theorem~\ref{thm:computeMults} and the induction argument in the proof of \cite[Proposition 6.18]{berkesch} reduces the computation of 
\[
\mu_{A^L,1}^{F,\nothing}(P^\beta,\beta) 
	- \mu_{A^L,0}^{F,\nothing}(P^\beta,\beta)
	\qquad 
\left(\text{and respectively } \mu_{A,1}^{L,\nothing}(P^\beta,\beta) - \mu_{A,0}^{L,\nothing}(P^\beta,\beta)\right)
\] 
to that of $\mu_{A^L,q}^{F,\nothing}(N,\beta)$ (and respectively $\mu_{A,q}^{L,\nothing}(N,\beta)$) for $q\geq 0$ and simple toric modules 
	$N=P_G^\beta$ with $\EE_G^{\beta}\neq \nothing$. 
Thus, by Lemma \ref{lemma-mult-convex},
\[
\mu_{A^L,1}^{F,\nothing}(P^\beta,\beta) 
	- \mu_{A^L,0}^{F,\nothing}(P^\beta,\beta)
  = \mu_{A,1}^{L,\nothing}(P^\beta,\beta) 
	- \mu_{A,0}^{L,\nothing}(P^\beta,\beta)
  = \mu_{A,0}^{L,\nothing}(\beta) 
    - \mu_{A}^{L,\nothing}
\]
which yields the desired equality.

Finally, since $S_A$ is a toric $S_{A^L}$-module,  $\HH_{0}^{A^L}(S_A,b)$ is a holonomic family by ~\cite[Theorem 7.5]{MMW}. 
Hence ~\cite[Theorem 2.6]{MMW} guarantees that
\[
\beta\mapsto \rank\left( \HH_0^{A^L}(S_A,\beta)\right)
\] 
is an upper semicontinuous function. 
\end{proof}

Theorem~\ref{thm:upper-semicont-convex-case} provides a way to prove Conjecture~\ref{conj:exceptional-Ltau} 
when $L$ is a convex filtration of $D$ with respect to $A$ and $\tau=\emptyset$. 

\begin{cor}
\label{cor:EAconvex}
If $L$ is a convex filtration of $D$ with respect to $A$, then 
\[
\cE_A^{L,\nothing} = 
-\,\qdeg\left(\bigoplus_{q=0}^{d-1}\Ext_{\CC[\del]}^{n-q}(S_A, \CC[\del])(-\varepsilon_A)\right).
\]
\end{cor}
\begin{proof}
By the proof of Theorem~\ref{thm:upper-semicont-convex-case}, 
$\HH_{0}^{A^L}(S_A,b)$ is a holonomic family and $\cE_A^{L,\nothing}=\cE_{A^L}^{F,\nothing}$, and thus by \cite[Theorem 9.1]{MMW},
\[
\cE_A^{L,\nothing} = 
\overline{\deg\left(\bigoplus_{i=0}^{d-1}H_{\mathfrak{m}_L}^i(S_A)\right)}^\text{Zariski},
%
\]
where $\mathfrak{m}_L$ denotes the maximal homogeneous ideal in $S_{A^L}$. 
However, since 
$\RR_{\geq 0} A= \RR_{\geq 0} A^L$, 
the radical of the extended ideal $\mathfrak{m}_L S_A$ in $S_A$ equals $\mathfrak{m}$.
Therefore, by applying graded Matlis duality, we obtain the desired result. 
\end{proof}

Let $L$ be a filtration on $D$ induced by a projective weight vector. For $\tau \in \Phi_A^{L}$, we denote by 
$A^{L,\tau}$ the submatrix of $A$ whose columns belong to facets 
$\tau' \in \Phi_A^{L,d-1}$ such that $\tau \subseteq \tau'$. We say that $L$ is $\tau$-\emph{convex} if 
all facets of $\Phi_A^L$ containing $\tau$ are $F$-homogeneous and the polytope
\begin{equation} 
\bigcup_{\tau \subseteq \tau' \in \Phi_A^{L,d-1}}
\Delta_{\tau '}  \label{convex-filtration}
\end{equation}
is convex, and thus equal to $\Delta_{A^{L,\tau}}$.

We recall that a subset $\eta '\subseteq A$ is said to be a \emph{pyramid} over $\eta\subseteq\eta'$ if  
\[
\rank_{\ZZ} (\ZZ \eta )+|\eta'\setminus \eta|=d,
\] 
where we denote by $|\lambda|$ the cardinality of a set $\lambda$.

Theorem \ref{thm:upper-semicont-convex-case} can now be generalized as follows.

\begin{cor}\label{cor: upper-semicont-tau-convex-case}
If $L$ induces a $\tau$-convex filtration for some $\tau \in \Phi_A^{L}$ and any $\tau ' \in \Phi_A^{L,d-1}$ such that $\tau\subseteq \tau'$ is a pyramid over $\tau '\setminus \tau$, then 
\[
\mu^{L,\tau}_{A ,0} (\beta )
  =\rank\left( \HH^{A^{L,\tau}}_{0}(S_A , \beta)\right). 
\]
In particular, $\mu^{L,\tau}_{A ,0} (\beta )$ is upper-semicontinuous in $\beta$. 
\end{cor}
\begin{proof}
Recall the formula in Theorem \ref{thm:swFormula}. For any $\tau '\in \Phi_A^L$ containing $\tau$, since $\tau '$ is a pyramid over $\tau '\setminus \tau$, 
it follows that 
\[\ZZ \tau'\cap \QQ \tau =\ZZ \tau,
\quad 
\pi_{\tau, \tau '}(\ZZ \tau ' )=\ZZ (\tau ' \setminus \tau), 
\quad 
P_{\tau,\tau'}=\Delta_{\tau '\setminus \tau}, 
\]
$Q_{\tau,\tau'}$ is the convex hull of $\tau ' \setminus \tau$ (whose volume is zero because $\tau'$ is $F$-homogeneous), 
and $\vol_{\ZZ \tau '}(\tau ')=\vol_{\ZZ (\tau ' \setminus \tau )}(\tau ' \setminus \tau )$. Thus, for any face $G\preceq A$ that contains $\tau$, 
\[
\mu^{L,\tau}_G = \sum_{\tau\subseteq\tau'\in\Phi^{L,d-1}_G}
    [\ZZ G :\ZZ\tau'] \cdot \vol_{\ZZ \tau '}\left(\Delta_{\tau '}\right) 
    =  \vol_{\ZZ G}\Bigg(\bigcup_{\tau\subseteq\tau'\in\Phi^{L,d-1}_G}\Delta_{\tau '}\Bigg)
    =\vol_{\ZZ G}\left(\Delta_{G^{L,\tau}}\right).
\]    
When $(G,b)\in\mathcal{J}(\beta)$, to obtain the equality 
\begin{equation}\label{eqn:rank-volume-tau}
  \rank\left( \HH_{0}^{G^{L,\tau}}(P^\beta_{(G,b)},\beta)\right)=\vol_{\ZZ G}\left(\Delta_{G^{L,\tau}}\right) 
\end{equation} 
we can proceed as in the proof of Proposition \ref{prop:mu-versus-rank}, but now $\RR_{\geq 0} G$ is not equal to $\RR_{\geq 0} G^{L,\tau}$, 
so $S_A$ is only a direct sum of weakly toric $S_{A^{L, \tau}}$-modules (by Remark \ref{remark-ranking-lattices}) instead of a toric $S_{A^{L,\tau}}$-module.
On the other hand, in the proof of Theorem \ref{thm:upper-semicont-convex-case} we can use $P_{J}^{\beta}$ with $J=\{(G,b)\in \cJ (\beta )|\, \tau \subseteq G \}$ instead of $P^{\beta}$ and consider each 
$P_G^{\beta}$ as a direct sum of weakly toric Cohen--Macaulay $S_{G^{L,\tau}}$-modules.

Finally, by \cite[Remark 5.5.(5)]{ekdi}, in the analytic topology, $\HH_{0}^{A^{L,\tau}}(S_A,\beta)$ is locally a holonomic family on $\AA^d$. This fact along with 
    ~\cite[Theorem 2.6]{MMW} and Theorem \ref{thm:computeMults} imply that the function $\beta \mapsto \rank\left( 
    \HH_{0}^{A^{L,\tau}}(S_A ,\beta)\right)$ is upper-semicontinuous. 
\end{proof}

\section{Gevrey series solutions associated to slopes}
\label{sec:gevrey}

Let $\DD$ be the sheaf of linear partial differential operators with coefficients in the sheaf $\cO_X^\text{an}$ of holomorphic functions on $X = \CC^n$. 
The irregularity sheaf of order $s>1$ of a 
holonomic $\DD$-module $\MM$ along a hypersurface $Y$ was introduced and proved to be a perverse sheaf on $Y$ by Mebkhout~\cite{Mebkhout}. In particular, higher cohomology of the irregularity sheaf vanishes at generic points of $Y$.

In this section, for a coordinate hyperplane $Y\subset X$, 
we compute the dimension of the stalk at a generic point $p\in Y$ of the irregularity sheaf of order $s$ of $\MM_A (\beta )\defeq \DD \otimes_D M_A (\beta )$ along $Y$ for any parameter $\beta \in \CC^d$, 
generalizing results from~\cite{maria-irregular}. 
As a consequence, we provide some formulas for the dimension of the Gevrey solution spaces of $\MM_A (\beta)$ in particular cases, and we show that the dimension of the generic stalk of the irregularity sheaf of $\MM_A (\beta)$ along $Y$ is upper-semicontinuous 
in $\beta$. 

We assume for simplicity that $Y=\Var(x_n)$ and write $s$ instead of $L(s)$ for the filtration given by 
$L(s)\defeq F+(s-1)V_n$ 
with $s\geq 1$, where 
$F=(\boldzero_n,\boldone_n)$ is the filtration by the order of the differential operators and $V_n$ is the Kashiwara--Malgrange filtration along $Y$. Recall that this filtration is induced by the projective weight vector $V_n \defeq (0,\ldots, 0, -1,0,\ldots ,0,1)$, where $-1$ is the weight for the variable $x_n$. More precisely, the filtration $L(s)$ is determined by
\[ 
\deg_s \partial_i 
  = \left\{\begin{array}{ll}
    1 & \mbox{if $1\leq i\leq n-1$,}\\
    s & \mbox{if $i=n$,}
    \end{array}\right.
\qquad \text{and}\qquad 
\deg_s (x_i) = 1 - \deg_s (\partial_i). 
\]

In this section, we call the $(A,L(s))$-umbrella instead the $(A,s)$-umbrella, and we denote $\Phi_A^{s}\defeq \Phi_A^{L(s)}$ for $s\geq 1$.

A global version of Laurent's slope theory~\cite{Laurent} proceeds as follows. 
Let $M$ be a holonomic $D$-module. A number $s>1$ is said to be a
\emph{slope} of $M$ along $Y=\Var(x_n)$ if and only if the $s$-characteristic variety $\charVar^{s}(M)$ of $M$ along $Y$ is not
homogeneous with respect to the weight vector $F=(\boldzero_n,\boldone_n)$.

\begin{remark}\label{remark-slopes-umbrella}
Denote by $A'$ the submatrix of $A$ defined by the first $n-1$ columns and by $\Delta'$ the convex hull of the columns of $A'$ and the origin. 
Note that 
$a_n/s$ belongs to a hyperplane off the origin that contains a facet of $\Delta_{A'}$
if and only if
there exists a facet of the $(A,s)$-umbrella, in other words an element of $\Phi_A^{s,d-1}$, that is not $F$-homogeneous. Moreover, by \cite[Corollary 4.18]{slopes}, this condition holds if and only if 
$s>1$ is a slope of $M_A (\beta)$ along $\Var(x_n)$.
\end{remark}

Let $\cO_{\widehat{X|Y}}$ denote the formal completion of $\cO_X$ along $Y$. 
A \emph{germ} $f \in\cO_{\widehat{X|Y},p}$ with $p\in Y$ is a formal series
\[
f = \sum_{m=0}^{\infty} 
  f_{m}(x_1,\dots,x_{n-1})x_{n}^{m}
\] 
such that there exists some open subset $U\subseteq \CC^{n-1}$ so that $f_m$ is a holomorphic function in $U$ for all $m\geq 0$. 
The formal series $f\in\cO_{\widehat{X|Y},p}$ is said to be a \emph{Gevrey series} of order $s\in \RR$ along $Y$ at $p\in Y$ if the series 
\[
\rho_{s}^{\tau}(f)
  \defeq \sum_{m=0}^{\infty} 
  \frac{f_{m}(x_1,\dots,x_{n-1})}
  {(m!)^{s-1}} x_{n}^{m}
\] 
is convergent at $p$.
Moreover, if $\rho_{s'}^{\tau} (f)$ is not convergent at $p$ for any $s'<s$, then $s$ is said to be the \emph{Gevrey index} of $f$ along $Y$ at $p$. 
Denote by $\cO_{X|Y}(s)$ the subsheaf of $\cO_{\widehat{X|Y}}$ whose germs are Gevrey series of order $s$ along $Y$.

The \emph{irregularity sheaf} of a $\DD$-module
$\MM$ along $Y$ of order $s>1$ is
\[
\Irr_Y^{(s)}(\MM)
  \defeq \RR\hspace{-.35ex}\Hom_{\DD}
  (\MM,\cO_{\widehat{X|Y}}(s)/\cO_{X|Y}).
\]
For $s=\infty$, the sheaf $\Irr_Y^{\infty}(\MM)$ is simply called the \emph{irregularity sheaf} of $\MM$ along $Y$. If $M$ is a $D$-module, we define $\Irr_Y^{(s)}(M)\defeq\Irr_Y^{(s)}(\MM)$, where $\MM\defeq \mathcal{D}\otimes_D M$. 

Set $\ds(A,\beta) \defeq \dim H^{0}(\Irr_{Y}^{(s)}(M_A(\beta))_p)$ for a generic point $p\in Y=\Var(x_n)$. 
Applying Th\'eor\`eme 2.3.1 and (2.3.1) in \cite{LM} to this setting yields the equality 
\begin{equation}
\ds(A,\beta) 
 = \mu_{A,0}^{s+\epsilon,\nothing}(\beta) 
  -\mu_{A,0}^{1+\epsilon,\nothing}(\beta) 
  +\mu_{A,0}^{1+\epsilon,\{n\}}(\beta) 
  -\mu_{A ,0}^{s+\epsilon,\{n\}}(\beta)
\label{eqn:dimIrreg}
\end{equation} 
for $\epsilon >0$ small enough. In particular, if $\beta$ is not rank--jumping for $A$, then by Theorem~\ref{thm:swFormula} and  \cite[Theorem~7.5]{maria-irregular}, $\ds(A,\beta)$ is equal to 
\begin{equation}
\label{eqn:genericDimIrreg}
\ds (A)\defeq\mu_{A}^{s+\epsilon,\nothing}
  -\mu_{A}^{1+\epsilon,\nothing}
  +\mu_{A}^{1+\epsilon,\{n\}} 
  -\mu_{A }^{s+\epsilon,\{n\}}=
   \sum_{n\notin\tau \in
    \Phi_A^{s+\epsilon,d-1}\setminus\Phi_A^{1+\epsilon,d-1}} 
    \vol_{\ZZ^d}(A_{\tau}).
\end{equation}

\begin{remark}\label{remark-ds-faces}
Notice that \eqref{eqn:genericDimIrreg} also holds for any face $G\preceq A$ in place of $A$ when $a_n\in G$. Moreover,  
$\ds (G)=\dim H^{0}(\Irr_{Y}^{(s)}(M_G (\beta ' ))_p $ for a generic point $p\in Y'=\Var(x_n)\subseteq \CC^G$ 
and $\beta ' \in \CC G$ that is not rank--jumping for $G$. 
The genericity condition on $p$ requires that it avoids any other irreducible component of the singular locus of $M_G (\beta')$ (which is independent of $\beta'$ as a consequence of Theorem \ref{thm:sw char}). 
On the other hand, if $a_n\notin G$, then the coordinates indexed by $G$ of the projective weight vectors $L(s)$ and $F$ are the same. Hence the two induced filtrations over (any cyclic module over) the Weyl algebra in the variables indexed by $G$ are also the same. Thus, 
$\mu_{G}^{1+\epsilon,\tau}=\mu_{G}^{s+\epsilon ,\tau}$ for $\tau = \{n \}$ and $\tau =\nothing$ in this case, so $\ds (G)=0$. 
\end{remark}

\begin{prop}\label{lower-bound-ds}
For any $\beta\in \CC^d$, there is a lower bound $\ds (A,\beta)\geq \ds (A)$.
\end{prop}
\begin{proof}
For a $\ZZ^d$-graded $\CC[\del]$-module $N$, define $\operatorname{d}_s^{(j)}(N,\beta) \defeq \dim H^{0}(\Irr_{Y}^{(s)}(\HH_j (N, \beta))_p) $ for a generic point $p\in Y=\Var(x_n)$.  Then by the same argument as in \eqref{eqn:dimIrreg}, 
\begin{equation}
\operatorname{d}_s^{(j)}(N,\beta) 
 = \mu_{A,j}^{s+\epsilon,\nothing}(N,\beta) 
  -\mu_{A,j}^{1+\epsilon,\nothing}(N,\beta) 
  +\mu_{A,j}^{1+\epsilon,\{n\}}(N,\beta) 
  -\mu_{A ,j}^{s+\epsilon,\{n\}}(N,\beta)
\label{eqn:dimIrregN}
\end{equation} 
for $\epsilon >0$ small enough. Notice that $\ds (A,\beta)=\operatorname{d}_s^{(0)}(S_A,\beta)$. By \cite[Corollary 4.13]{slopes} and \eqref{eqn:dimIrregN}, $\operatorname{d}_s^{(0)}(\widetilde{S}_A ,\beta)=\ds(A)$. Moreover, if $j\geq 1$, then $\operatorname{d}_s^{(j)}(\widetilde{S}_A ,\beta)=0$ because 
$\HH_j (\widetilde{S}_A, \beta)=0$ by Theorem \ref{thm:higherHHvanishing}. 

On the other hand, if $N$ is a toric module with dimension lower that $d$, it follows that 
\[
\operatorname{d}_s^{(0)}(N,\beta)\leq \operatorname{d}_s^{(1)}(N,\beta)
\] 
by the same argument as in the proof of \cite[Lemma 4.29]{slopes}, with the replacement, for each $D$-module $M$ that appears in that proof, of the role of $\charC^L(M)$ by $\dim H^{0}(\Irr_{Y}^{(s)}(M)_p) $ for a generic point $p\in Y=\Var(x_n)$. This is allowable because $H^{1}(\Irr_{Y}^{(s)}(M)_p)=0$ for generic points $p\in Y$ when $M$ is holonomic (see \cite{Mebkhout}). Thus, with the previous ingredients, the proof of \cite[Theorem 4.28]{slopes} gives the result with $\operatorname{d}_s^{(0)}$ in place of $\mu_{A,0}^{L,\tau}$.\end{proof}

\begin{cor}
\label{cor:dimStalkIrr}
For $s>1$, the dimension $\ds (A,\beta)$ of the stalk of 
$\Irr_Y^{(s)}(M_A (\beta ))$ at a generic
point $p$ of $Y$ can be computed from the combinatorics of
$\Phi_A^{s+\epsilon}\setminus \Phi_A^{1+\epsilon}$ for $\epsilon
>0$ small enough and the ranking lattices $\bbE_G^{\beta}$ at
$\beta$ such that $a_n \in G\preceq A$.
\end{cor}
\begin{proof}
It follows from \eqref{eqn:dimIrreg} and Theorem~\ref{thm:computeMults} that $\ds (A,\beta)$ can be computed from the combinatorics of the $(A,s')$-umbrellas for $s'\in\{1+\epsilon, s + \epsilon\}$ and the ranking lattices $\EE^{\beta}$. Thus, by Remark \ref{remark-ds-faces}, it is enough to consider the ranking lattices $\EE_G^{\beta}$ at $\beta$ corresponding to the faces $G\preceq A$ containing $a_n$.
\end{proof}

We now state further consequences for $\ds (A,\beta)$. 

\begin{cor}
\label{cor:simpleComputeGevrey}
If $P^{\beta}=P^{\beta}_G$ for some $G\preceq A$, then
\[
\ds (A, \beta ) =\ds (A) + |B^\beta_{G}|\cdot (\codim(G)-1) \cdot \ds (G).
\] In particular, if $a_n\notin G$ or $\codim(G)=1$, then $\ds (A,\beta)=\ds(A)$.
\end{cor}
\begin{proof}
It is a direct consequence of  \eqref{mu-Q}, \eqref{mu-P-Q}, Lemma \ref{lem:simpleRT}, \eqref{eqn:dimIrreg}, \eqref{eqn:genericDimIrreg}, and Remark \ref{remark-ds-faces}.
\end{proof}

\begin{cor}\label{cor:dim2Gevrey}
If $d=2$, then $\ds (A ,\beta)=\ds (A)$ for any $\beta\in \CC^d$.
\end{cor}
\begin{proof}
Since $d=2$, the matrix $A$ has only two proper faces $G_1, G_2 \preceq A$, which both have codimension $1$. Moreover, $a_n$ belongs to at most one of these two facets. Thus, by Corollaries \ref{cor:dimStalkIrr} and \ref{cor:simpleComputeGevrey}, it is enough to consider the case when $a_n\in G_1$ and $\max(\mathcal{J}(\beta))$ involves $G_1$. In this case, $\ds (A,\beta)$ can be computed as in the simple case, so the formula in Corollary \ref{cor:simpleComputeGevrey} can be applied, giving $\ds (A ,\beta)=\ds (A)$ since $\codim (G_1 )=1$.
\end{proof}

Notice that Corollary \ref{cor:dim2Gevrey} also follows from \cite[Proposition 4.25]{slopes} and  \eqref{eqn:dimIrreg}.

\begin{cor}\label{cor:d=3-Gevrey}
If $d=3$, then $\ds (A, \beta ) >\ds (A)$ if and only if $\operatorname{max}(\mathcal{J}(\beta))$ involves a face $G$ with $a_n\in G$ and $\dim G=1$. If this is the case, $\ds(A,\beta)=\ds (A) + |B^\beta_{G}| \cdot \ds (G)$.
\end{cor}
\begin{proof}
Again by Corollary \ref{cor:dimStalkIrr}, we only need to consider the ranking lattices $\EE_G^{\beta}$ such that $a_n\in G$. Thus, by the reduction given in  \cite[Section 5.3]{berkesch}, it is enough to prove the result in the following two cases. 

The first case is that $a_n$ belongs to a unique face $G$ among those involved in $\max(\mathcal{J}(\beta))$. In this case, the computation follows as in the simple case, and we obtain the same formula as in Corollary~\ref{cor:simpleComputeGevrey}. 

In the second case, we may assume that there are exactly two faces $G_1$ and $G_2$ involved in $\max(\mathcal{J}(\beta))$ that contain $a_n$. Since the face $G_1\cap G_2$ contains $a_n$ and $d=3$, it follows that $G_1$ and $G_2$ are two facets intersecting in a face of codimension $2$. In this case, Remark~\ref{ex:2cmpts} shows that $\mu_{A,0}^{L,\tau}(\beta)=\mu_A^{L,\tau}$ for any filtration $L$ and any $\tau\in \Phi_A^L$, so $\ds (A,\beta)=\ds (A)$.
\end{proof}

\begin{lemma}\label{upper-semicont-Gevrey-lemma}
Let $s> 1$ be such that the $(A,s)$-umbrella $\Phi_A^s$ has a unique facet $\tau$ that is not $F$-homogeneous, $p$ is a generic point of $Y=\Var(x_n)$, and $\epsilon>0$ small enough.  
Then the function 
\[
\beta \mapsto d (A,\beta,s)\defeq\dim \mathcal{H}om_{\mathcal{D}}(\mathcal{M}_A (\beta),\cO_{\widehat{X|Y}}(s+\epsilon)/\cO_{\widehat{X|Y}}(s -\epsilon))_p
\] 
is upper-semicontinuous.
\end{lemma}
\begin{proof}
Notice first that by the assumption and Remark \ref{remark-slopes-umbrella}, $s$ is a slope of $M_A (\beta)$ along $Y$ and $a_n \in \tau$. Indeed, the assumption implies that $\tau'\defeq\tau \setminus \{n\}$ is the unique facet of $\Phi_A^{s+\epsilon}$ that does not contain $a_n$ and is also not a facet of $\Phi_A^{s-\epsilon}$. 
On the other hand, 
\begin{align*}
d(A,\beta,s)&=d_{s+\epsilon} (A,\beta)-d_{s-\epsilon}(A,\beta)\\
&=\mu_{A,0}^{s+\epsilon,\nothing}(\beta) 
  -\mu_{A,0}^{s-\epsilon,\nothing}(\beta) 
  +\mu_{A,0}^{s-\epsilon,\{n\}}(\beta) 
  -\mu_{A ,0}^{s+\epsilon,\{n\}}(\beta).
\end{align*}
Thus, setting $d(A,s)\defeq\mu_{A,0}^{s+\epsilon,\nothing} 
  -\mu_{A,0}^{s-\epsilon,\nothing} 
  +\mu_{A,0}^{s-\epsilon,\{n\}} 
  -\mu_{A ,0}^{s+\epsilon,\{n\}}$ yields
\begin{align*}
d(A,s)=\vol_{\ZZ A}(\Delta_{\tau'})=\operatorname{rank}(\HH^{\tau ' }_0 ( \widetilde{S}_A ,\beta)), 
\end{align*}
where the first equality follows by the assumption, \eqref{eqn:F-homogeneous-mult}, and \cite[Lemma 7.4]{maria-irregular}. The second equality follows as in the proof of \eqref{eqn:rank-volume-tau}, since $A_{\tau'}$ is a rank $d$ submatrix of $A$. Similarly, for faces $G$ of $A$ such that $a_n \in G$ and $\tau'':=\tau'\cap G$ is a facet of $\Phi_G^{s+\epsilon}$, we also have that $d(G,s)=\operatorname{rank}(\HH^{\tau''}_0 ( \widetilde{S}_G ,\beta))$. Thus, arguments similar to those in Corollary \ref{cor: upper-semicont-tau-convex-case} show that $d(A,\beta,s)=\operatorname{rank}(\HH^{\tau'}_0 (S_A ,\beta))$
and that the function $\beta \mapsto d (A,\beta,s)=\operatorname{rank}(\HH^{\tau'}_0 (S_A ,\beta))$ is upper-semi\-continuous in $\beta$.
\end{proof}

\begin{theorem}
Assume that for all $s>1$, $a_n/s$ is in at most one of the hyperplanes off the origin supported in a facet of $\Delta'$ (see Remark \ref{remark-slopes-umbrella}). 
Then the function $\beta \mapsto \ds (A,\beta)$ is upper-semicontinuous for all $s>1$.
\end{theorem}
\begin{proof}
Let $1<s_1 < \cdots < s_r$ be the set of slopes of $M_A (\beta)$ along $Y$ that are lower or equal to $s$. Then 
$d_s (A,\beta)=\sum_{j=1}^r d ( A, \beta , s_j )$, and the result follows by Lemma \ref{upper-semicont-Gevrey-lemma}.
\end{proof}

In view of the preceding results we state the following conjecture.

\begin{conjecture}
\label{conj:exceptional-ds}
The map $\beta \mapsto \ds (A,\beta)$ is upper-semicontinuous. Moreover, there is an equality 
\[
\cE_A^{n}(s)\defeq\{\beta\in\CC^d\mid  \ds (A,\beta)> \ds (A) \}=-\,  \operatorname{qdeg}
  \left(
  \bigoplus_{q=0}^{d-1} \Ext^{n-q}_{\CC[\del]} (
  S_A^{\{n\}},\CC[\del])
  (-\varepsilon_A)\right),
\] 
where $\varepsilon_A \defeq \sum_{i=1}^n a_i$. In particular, $\cE_A^{n}(s) =\nothing$ if and only if $S_A^{\{n\}}$ is Cohen--Macaulay.
\end{conjecture}

The values $\ds(A,\beta)$ and $\ds (A)$ defined in this section depend on the variety $Y$ along which we are considering the irregularity sheaf of $M_A(\beta)$. Although we assumed $Y=\Var(x_n)$ for simplicity, 
we can consider any $Y_j:=\Var(x_j)\subseteq\CC^n$ since reordering the variables is equivalent to reordering the columns of $A$. Let $\ds(A,\beta,j)$ and $\ds(A,j)$ denote the values of $\ds(A,\beta)$ and $\ds (A)$ respectively for $Y_j$ in place of $Y$.  In the following example, we compute the difference $\ds (A,\beta ,j)-\ds (A,j)$ for different $j$ by using Corollary \ref{cor:d=3-Gevrey}.

\begin{example}
Let us consider the matrix $A$ in Example \ref{ex:Lmatters-d=3}. The hyperplanes contained in the singular locus of $M_A(\beta)$ are exactly $Y_j$ for $j\in\{2,4,6,7\}$ and there is exactly one slope $s_j\geq 1$ of $M_A(\beta)$ along each $Y_j$. 
More precisely, by Remark \ref{remark-slopes-umbrella}, $s_2=3/2$, $s_4=3$, $s_6=2$, and $s_7=7/6$. It is clear that $\ds(A,\beta',j)=0$ if $1\leq s <s_j$ for any $\beta'\in \CC^d$, so let us assume that $s\geq s_j$ in each case.
We have that $\ds (A,\beta',j)=\ds (A,j)$ for all $\beta'$ and $j\in\{2,7\}$. On the other hand, $\ds(A,\beta',4)-\ds(A,4)$ is $1$ if $\beta'\in \beta+\CC G_1$ and $0$ otherwise.
Finally, $\ds(A,\beta',6)-\ds(A,6)$ is $1$ if $\beta'\in\beta+\CC G_2$ and zero otherwise.
\end{example}

One natural problem after the computation of $\ds(A,\beta)=m$ is to construct an explicit set of Gevrey series $\varphi_1,\ldots,\varphi_m$ along $Y$ at a nonsingular point $p\in Y$ so that their classes in the space $(\cO_{\widehat{X|Y}}(s)/\cO_{X|Y})_p$ form a basis of $H^{0}(\Irr_{Y}^{(s)}(M_A(\beta))_p)$. This was done in \cite{maria-irregular} when $\beta$ is generic enough. At any parameter $\beta$, this problem is much more involved in general. However, it is easy to compute some examples by using a slightly modified version of a method used in \cite{Fer-exp-growth}. 
In order to do so, recall that the direct sum of two matrices $A_1 \in \ZZ^{d_1 \times n_1}, A_2
\in \ZZ^{d_2 \times n_2}$ is the following $(d_1 + d_2 )\times (n_1
+ n_2)$ matrix:
\[
A_1 \oplus A_2 = \left(
\begin{array}{cc}
A_1 & \boldzero_{d_1 \times n_2} \\                                                        \boldzero_{d_2 \times n_1 } & A_2 \end{array}
\right),
\] 
where $\boldzero_{d\times n}$ denotes the $d\times n$ zero matrix. Let $\beta=(\beta^{(1)},\beta^{(2)})$ denote a complex vector in $\CC^{d_1+d_2}\cong\CC^{d_1}\times\CC^{d_2}$. It is easy to show using \cite[Lemma 2.2]{Fer-exp-growth} that 
\[
\ds(A,\beta,n_1)=\ds(A_1,\beta^{(1)},n_1)\cdot\rank(M_{A_2} (\beta^{(2)})). 
\]
Now, let us take $(A_1,\beta^{(1)})$ such that $M_{A_1}(\beta^{(1)})$ has slopes along $\{x_{n_1}=0\}$, and let consider the subset of Gevery series $\{g_1,\ldots,g_{r(1)}\}\subseteq \cO_{\widehat{X|Y}}(s)$ whose classes form a basis of 
\[
H^{0}(\Irr_{\{x_{n_2}=0\}}^{(s)}(M_{A_1}(\beta^{(1)}))_p). 
\]
Let us take also a pair $(A_2,\beta^{(2)})$ for which a basis $\{f_1,\ldots,f_{r(2)}\}$ of convergent series solutions of $M_{A_2}(\beta^{2})$ at a nonsingular point $p'$ is known for a rank--jumping parameter $\beta^{(2)}\in\CC^{d_2}$. Then $\{g_i f_j \mid 1\leq i\leq r(1),\; 1\leq j\leq r(2) \}$ is a basis of 
\[
H^{0}(\Irr_{\{x_{n_2}=0\}}^{(s)}(M_{A_1 \oplus A_2}(\beta))_{(p,p')}),
\] 
where $\beta=(\beta^{(1)},\beta^{(2)})$. Note that  
\[
\dim H^{0}(\Irr_{\{x_{n_2}=0\}}^{(s)}(M_{A_1 \oplus A_2}(\beta))_{(p,p')})
 = 
r(1)\cdot r(2)>\ds(A,n_1)=\ds(A_1,n_1)\cdot \mu_{A_2,0}^F. 
\]
In particular, the smallest example of this family is the one obtained by taking $A=A_1\oplus A_2$ for $A_1=(1\; 2)$ and $A_2=(\widetilde{0},\widetilde{1},\widetilde{3},\widetilde{4})$, where $\widetilde{a}=(1,a)^t$ and $\beta=(b,1,2)^t$ for any $b\in\CC\setminus\ZZ$. We notice that $M_{A_2}((1,2)^t)$ was the first example known of an $A$--hypergeometric system for which the rank is greater than the normalized volume~\cite{ST98}. Indeed, a  basis of $H^{0}(\Irr_{\{x_{n_2}=0\}}^{(s)}(M_{A_1}(b))_p)$ is $\{\overline{\phi_v} \}\subset (\cO_{\widehat{X|Y}}(s)/\cO_{X|Y})_p$, where $\phi_v$ is the $\Gamma$--series associated to $v=(b,0)$ (see \cite{maria-irregular}) and $\rank (M_{A_2}(\beta^{(2)}))=\vol_{\ZZ^2}(A_2)+1=5$ (see \cite{ST98}, where a basis of solutions is also described). Thus, in this case, $M_A (\beta)$ has the slope $s=2$ along $x_2=0$ and for $s\geq 2$, $\ds (A,\beta,2)=\ds(A,2)+1=5$.  

\raggedbottom
\def\cprime{$'$} \def\cprime{$'$}
\providecommand{\MR}{\relax\ifhmode\unskip\space\fi MR }
\providecommand{\MRhref}[2]{%
  \href{http://www.ams.org/mathscinet-getitem?mr=#1}{#2}
}
\providecommand{\href}[2]{#2}

\end{document}